\documentclass[a4paper,12pt,reqno]{amsart}
\usepackage[body={17cm,21cm}]{geometry}
\usepackage[initials]{amsrefs}

\usepackage{amssymb,amscd}
\usepackage[all]{xy}

\usepackage[mathscr]{eucal}
\usepackage{enumerate}

\numberwithin{equation}{section}
\theoremstyle{plain}
\newtheorem{thm}[equation]{Theorem}

\newtheorem{lem}[equation]{Lemma}
\newtheorem{prop}[equation]{Proposition}

\newtheorem{definition}[equation]{Definition}
\newtheorem{exa}[equation]{Example}
\newtheorem{rem}[equation]{Remark}

\theoremstyle{definition}

\theoremstyle{remark}

\providecommand{\R}[1]{\mathrm{#1}}

\DeclareMathOperator{\Br}{Br}

\DeclareMathOperator{\Gal}{Gal}

\DeclareMathOperator{\Hom}{Hom}

\DeclareMathOperator{\Pic}{Pic}

\DeclareMathOperator{\Spec}{Spec}

\def\br{   {\rm Br\,}       }

\def\Br{{\rm Br\,}}

\def\Q{{\mathbb Q}}
\def\Z{{\mathbb Z}}
\def\R{{\mathbb R}}

\def\N{{\mathbb N}}

\def\A{{\mathbb A}}
\def\Gal{{\rm Gal}}

\def\ov{\overline}
\def\br{{\rm Br}}

\def\Pic{{\rm Pic}}

\def\oi{\hskip1mm {\buildrel \simeq \over \rightarrow} \hskip1mm}

\newcommand{\et}{{\textup{\'et}}} 

\DeclareFontFamily{U}{wncy}{}
\DeclareFontShape{U}{wncy}{m}{n}{%
<5>wncyr5%
<6>wncyr6%
<7>wncyr7%
<8>wncyr8%
<9>wncyr9%
<10>wncyr10%
<11>wncyr10%
<12>wncyr6%
<14>wncyr7%
<17>wncyr8%
<20>wncyr10%
<25>wncyr10}{}
\DeclareMathAlphabet{\cyr}{U}{wncy}{m}{n}
\begin{document}

\title[Strong approximation for the total space of certain quadric fibrations]
{Strong approximation for the total space of certain quadric fibrations}

\author{J.-L. Colliot-Th\'el\`ene}
\author{Fei XU}

\address{CNRS, UMR 8628, Math\'ematiques, B\^at. 425, Univ. Paris-Sud, F-91405 Orsay, France}

\email{jlct@math.u-psud.fr}

\address{School of Mathematical Sciences, Capital Normal University, Beijing 100048, China}

\email{xufei@math.ac.cn}

\date{December 12th,  2011}

\maketitle

\section{Introduction}

Let $X$ be a variety over a number field $F$.
For simplicity, let us assume in this introduction that the set $X(F)$ of rational points is
not empty.
Let $S$ be a finite set of places of $F$.
One says that {\it strong approximation holds for $X$ off $S$}
if the diagonal image of the set $X(F)$ of rational points  is dense in
  the space of $S$-ad\`eles $X(\A_F^S)$ (these are the ad\`eles where
  the places in $S$ have been omitted) equipped with the adelic topology.
    If this property holds for  $X$, it in particular implies
a local-global principle for the existence of integral points on
integral models of $X$ over the ring of $S$-integers of $F$.

For $X$ projective, $X(\A_F^S) = \prod_{v \notin S} X(F_v)$,
and the adelic topology  
coincides with the product topology.
A projective variety satisfies strong approximation off $S$
if and only if  weak approximation
 for the rational points holds off $S$.

For   open varieties, strong approximation has been mainly studied
for linear algebraic groups and their homogeneous spaces.
A classical case is $m$-dimensional affine space ${\bf A}^m_{F}$ off any nonempty set $S$, a special case being the Chinese
Remainder Theorem. For a semisimple,  almost simple, simply connected  linear algebraic group $G$
such that $\prod_{v \in S} G(F_v)$ is not compact, strong approximation
off $S$ was established by   Eichler, Kneser, Shimura,  Platonov, Prasad.

Strong approximation 
does   not hold for groups which are
not simply connected, but one may define a Brauer-Manin set. In our paper  \cite{CTX}, we started the investigation of the Brauer-Manin obstruction to strong approximation  for  homogeneous spaces of linear algebraic groups.  For such varieties, 
this was quickly  followed by  works of  Harari \cite{H}, Demarche \cite{Dem},
 Borovoi and Demarche \cite{BD} and Wei and Xu \cite{WX1}.

Few strong approximation results are known 
 for open varieties which are not homogeneous spaces. 
 Computations of the Brauer-Manin obstruction for some such varieties  
 have been recently performed (Kresch and Tschinkel  \cite{KT}, Colliot-Th\'el\`ene et Wittenberg \cite{CTW}).
Much earlier,  Watson \cite{Wat} investigated integral points on affine varieties which are
the total space of families of quadrics  over affine space ${\bf A}^m_{F}$.
When restricted to equations as above, in particular $m=1$, and with coefficients
in the ring $\Z$ of integers, under a non-compacity assumption, his Theorems 1 and 2
establish the local-global principle for  integral points when $n \geq  4$ (\cite[Thm. 1, Thm.2]{Wat}).
 Under some additional condition, he also establishes a local-global principle when $n=3$ (\cite[Thm. 3]{Wat}, see Remark \ref{watson} in the present paper).

Just as for problems of weak approximation, it is natural to ask whether strong
approximation holds for the total space of a family $f: X \to Y$ when 
it is known for the basis $Y$, for many fibres of $f$, and some algebraico-geometric assumption
is  made on the map $f$.

In the present paper, we   investigate strong approximation for varieties $X/F$ defined by an equation
$$q(x_1,\dots,x_n)=p(t),$$
where  $q$ is a quadratic form of rank $n$  in $n \geq 3$ variables and $p(t) $ is a nonzero polynomial.

The paper is organized as follows. In \S \ref{basic} we recall some definitions
related to strong approximation and the Brauer-Manin obstruction.

  In \S \ref{general},
 we give a simple  general method for proving   strong approximation for the total space of a fibration.
We apply it to varieties defined  by an equation $q(x_1,\dots,x_n)=p(t),$
for $n \geq 4$.  

In   \S  \ref{basicquadric} we detail results of \cite{CTX} on 
  the arithmetic of affine quadrics $q(x,y,z)=a$. 
  
In the purely algebraic  \S \ref{computBrauer}, we compute the Brauer group of   the smooth locus, and of  a suitable desingularisation, of a
variety  defined by an equation
$q(x,y,z)=p(t).$

The most significant results are given in \S \ref{2dim}.
The results of \S \ref{basicquadric}  and  \S \ref{computBrauer} are combined
to study the strong approximation property off $S$  for certain smooth models of varieties defined by
  an equation
$q(x,y,z)=p(t),$
under the assumption that the form $q$ is isotropic at some place in~$S$.
For these smooth models, when there is no Brauer-Manin obstruction,  we establish   strong approximation off $S$.
We give the precise conditions under which strong approximation fails.

In \S \ref{twoexamples} we give two numerical counterexamples to the local-global principles 
for existence of integral points: this represents a drastic failure of strong approximation in the
cases where this is allowed  by the results of \S \ref{2dim}.

  Concrete varieties often are singular. In that case the appropriate properties are  ``central strong approximation" and its Brauer-Manin variant. This is shortly discussed in \S \ref{singular}.

\section{Basic definitions and properties}\label{basic}

Let $F$ be a number field, $\frak o_F$ be the ring of integers of
$F$ and $\Omega_F$ be the set of all primes in $F$. For each $v \in \Omega_F$, let $F_v$ be the completion of $F$ at $v$.
Let  $\infty_F$ be  the set of  archimedean primes in $F$ and write
$v<\infty_F$ for $v\in \Omega_F\setminus \infty_{F}$.
For each $v<\infty_F$, let $\frak o_{v}$ 
be the completion 
of $\frak o_F$ at $v$ and let $\pi_v$ be a uniformizer of
$\frak o_{v}$. Write $\frak o_{v}=F_v$ for
$v\in \infty_F$.

For any finite subset $S$ of $\Omega_F$, let $F_{S}=\prod_{{v} \in S}F_{v}$.
For any finite subset $S$ of $\Omega_F$ containing $\infty_F$,
 the $S$-integers are defined to be elements in $F$ which are integral outside $S$. The ring of $S$-integers is  
 denoted by $\frak o_S$. Let $\Bbb A_F \subset \prod_{v \in \Omega_{F}}F_v $ be the adelic group of $F$
with its usual topology.
For any finite subset $S$ of $\Omega_F$, one   defines 
$\Bbb A_{F}^S \subset  (\prod_{v\not \in S} F_v)$
equipped with  the analogous   adelic
  topology. The  natural projection which omits the  $S$-coordinates  defines a  homomorphism of rings
$  \Bbb A_F \to \Bbb A_{F}^S$. For any variety  $X$ over $F$ this induces a map
$$ pr^S :  X(\Bbb A_{F})   \to   X(\Bbb A_{F}^S)$$
which is surjective if  
$\prod_{v \in S} X(F_v) \neq \emptyset$.

\begin{definition}\label{defstrong}
Let $X$ be a geometrically integral $F$-variety. 
One says that strong approximation holds for   $X$ off $S$   if
the image of the diagonal map
$$ X(F)  \to X(\Bbb A_{F}^S)$$
is dense in  $pr^S(X(\Bbb A_{F}) ) \subset  X(\Bbb A_{F}^S)$.
\end{definition}

The statement may be rephrased as:

{\it Given any nonempty open set $W \subset X(\Bbb A_{F}^S)$, 
if $X(\A_{F})\neq \emptyset $,
then there is a point of the diagonal image of $X(F)$
which lies in $W \times \prod_{v \in S} X(F_{v}) \subset X(\A_{F})$.}

\medskip

If $X$ satisfies strong approximation off $S$, and  $X(\Bbb A_{F}^S)\neq \emptyset$,
then $X(F) \neq \emptyset$ and for any finite set $T$ of places  of $F$ away from $S$,
the diagonal image of $X(F)$ is dense in $\prod_{v\in T} X(F_v)$. In other words,
$X$ satisfies the Hasse principle, and  $X$ satisfies weak approximation off $S$.

\begin{prop}\label{strongSS'}
Assume $X(\Bbb A_{F})\neq \emptyset$.
If $X$ satisfies strong approximation off  a finite  set $S$ of places,
then it satisfies   strong approximation off  any  finite set $S'$ with $S \subset S'$. \qed
\end{prop}

\begin{prop}\label{strongUX}
Let $U \subset X$  be a  dense open set of a smooth geometrically integral $F$-variety $X$.
If strong approximation off $S$  holds for $U$, then strong approximation off $S$
holds for $X$.
\end{prop}
\begin{proof}
This follows from the following statement: for $X/F$ as in the lemma, the image of $U(\A_F)$
in $X(\A_F)$ is dense. That statement itself follows from two facts. Firstly, for a given place $v$,
$U(F_v)$ is dense in $X(F_v)$ (smoothness of $X$). Secondly, $U$ admits a model $\bf U$ over
a suitable $\frak o_T$ such that ${\bf U}(\frak o_v) \neq \emptyset$ for all $v \notin T$ (because $U/F$ is geometrically integral).
\end{proof}

As explained in \cite{CTX},
one can refine definition  \ref{defstrong} by using the Brauer--Manin set. Let $X$ be an $F$-variety.
Let $\br(X)=H^2_{\et}(X, \Bbb G_m)$ and define
$$X(\Bbb A_F)^{\br(X)}=\{\{x_v\}_{v\in \Omega_F}
\in X(\Bbb A_F): \ \forall \xi\in \br(X),  \ \sum_{v\in \Omega_F}inv_v (\xi(x_v)) =0.
      \} $$
This is a closed subset of $X(\Bbb A_F)$. Class field theory implies  
$$X(F) \subset  X(\Bbb A_F)^{\br(X)} \subset X(\Bbb A_F). $$ Let
$$ X(\Bbb A_F^S)^{\br(X)}:=  pr^S (X(\Bbb A_F)^{\br(X)}) \subset X(\Bbb A_{F}^S). $$

\begin{definition}\label{defstrongmanin}
Let $X$ be a geometrically integral variety over  the number field $F$. 
If  the diagonal  image of $X(F)$ in  $(X(\Bbb A_{F}^S))^{\br(X)} \subset X(\Bbb A_{F}^S)$ is 
 dense,
we say that   strong approximation with Brauer-Manin obstruction holds for $X $ off  $S$.
\end{definition}

As above, the statement may be rephrased as:

{\it Given any   open set $W \subset X(\Bbb A_{F}^S)$,
if  $[W \times \prod_{v \in S} X(F_{v})]^{\br(X)} \neq \emptyset $,  
 then there is a point of the diagonal image of $X(F)$ 
  in the product $W \times \prod_{v \in S} X(F_{v}) \subset X(\A_{F})$.}

\begin{prop}\label{strongBMSS'}
Assume $X(\Bbb A_{F})\neq \emptyset$.
If strong approximation with Brauer-Manin obstruction holds for $X $ 
off  a finite  set $S$ of places, then it holds off 
  any   finite set $S'$ with $S \subset S'$. \qed
\end{prop}

\begin{prop}\label{strongBMUX}
 Let $F$ be a number field.
Let $U \subset X$  be a  dense open set of a smooth geometrically integral $F$-variety $X$.
Assume:

(i)  $X(\Bbb A_{F})\neq \emptyset$;

(ii) the quotient $\br(U)/\br(F)$ is finite.

Let $S$ be a finite set of places of $F$. If strong approximation  with Brauer-Manin obstruction   off $S$ holds for $U$, then 
it holds for
$X$.
\end{prop}

\begin{proof} There exists a finite subgroup  $B \subset \br(U)$
such that $B$ generates
  $\br(U)/\br(F)$ and $B \cap \br(X)$ generates $\br(X)/\br(F)$.
There exists a finite set $T$ of places of $k$ containing $S$ and all the archimedean places,
and smooth $\frak o_{T} $-schemes $\bf{U} \subset {\bf X}$  with geometrically integral fibres over
the points of $\Spec(\frak o_{T})$
such that

(a) The restriction $\bf{U} \subset {\bf X}$ over $\Spec(F) \subset \Spec (\frak o_{T})$ is $U \subset X$.

(b) $B \subset \ \br({\bf U})$.

(c)  $B \cap \br(X) \subset \br({\bf X})$.

(d) For each $v \notin T$,  ${\bf U}(\frak o_{v}) \neq \emptyset$ (this uses the fact that $U \to 
\Spec(\frak o_{T})$ is smooth with geometrically integral fibres, the Weil estimates and the
fact that we took $T$ big enough).

To prove the proposition, it is enough to show: 

\smallskip

  {\it Given any finite set $T$ as above and  given, for each place $v \in T \setminus S$,
an open set $W_{v} \subset X(F_{v})$  such that  the set
$$[\prod_{v \in S} X(F_{v}) \times \prod_{v \in T\setminus S}W_{v} \times \prod_{v \notin T} {\bf X}(\frak o_{v})]^{\br(X)}   $$
is not empty, then  this set contains a point of the diagonal image   of $X(F)$ in $X(\A_{F})$.}
 
 \smallskip

  Each $\alpha \in B\cap \br(X)$ vanishes when evaluated on ${\bf X}(\frak o_{v})$.
 For any  $\alpha \in \br(X)$ and any place $v$,  the map $X(F_{v}) \to \br(F_{v}) \subset \Q/\Z$ 
 given by evaluation of $\alpha$ is locally constant.
 Since $X$ is smooth, for each place $v$, the set $U(F_{v})$ is dense in $X(F_{v})$ for the 
local topology. In particular, for $v \notin T$, the set ${\bf X}(\frak o_{v}) \cap U(F_{v})$
is not empty.
There thus exists a point $\{M_{v}\} \in X(\A_{F})$ which lies in the above set
such that $M_{v} \in U(F_{v})$ for $v \in T$ and $M_{v} \in {\bf X}(\frak o_{v}) \cap U(F_{v})$
for $v \notin T$.

 We now use Harari's formal lemma in the version given  in \cite{CT}.
 According to the proof of \cite[Th\'eor\`eme 1.4]{CT}, there 
  exists a finite set $T_{1}$ of places of $k$, $T_{1} \cap T = \emptyset$,
  and for $v \in T_{1}$ points $N_{v} \in {\bf X}(\frak o_{v}) \cap U(F_{v})$,
  such that
  $$ \sum_{v \in T} \beta(M_{v}) + \sum_{v \in T_{1}} \beta(N_{v})=0$$
  for each $\beta \in B$.
  
  For $v \in T$,   let $N_{v}=M_{v}$. For $v \notin T \cup T_{1}$,  let $N_{v} \in {\bf U}(\frak o_{v})$ be an arbitrary point.
The ad\`ele $\{N_{v}\}$  of $X$ belongs to
$$[\prod_{v \in S} X(F_{v}) \times \prod_{v \in T\setminus S}W_{v} \times \prod_{v \notin T} {\bf X}(\frak o_{v})]^{\br(X)}   $$
It is the image of  an ad\`ele of $U$ which lies in
$$[\prod_{v \in S} U(F_{v}) \times \prod_{v \in T\setminus S}W_{v}\cap U(F_{v}) \times \prod_{v\in T_{1}} U(F_{v}) \cap {\bf X}(\frak o_{v})
\times
 \prod_{v \notin T \cup T_{1}} {\bf U}(\frak o_{v})]^{\br(U)}.$$
 
 Using the finiteness of $B$ and the continuity of the evaluation map $U(F_{v}) \to \br(F_{v})$
 attached to each element of $B$, we find that there exist open sets $W'_{v} \subset U(F_{v})$
 for $v \in T\cup T_{1}$, with  $W'_{v} \subset W_{v}$ for $v \in T \setminus S$,
 such that the subset
$$[\prod_{v \in T\cup T_{1}} W'_{v} \times  \prod_{v \notin T \cup T_{1}} {\bf U}(\frak o_{v})]^{\br(U)}$$
 of the ad\`eles of $U$
 is nonempty. Since  strong approximation  with Brauer-Manin obstruction   off $S$ holds for $U$, 
 hence off $T \cup T_{1}$ since $S \subset T$, there exists 
a point in the diagonal image of $U(F)$ in $U(\A_{F})$ which lies
 in this set.
 
Since this set maps into 
 $$[\prod_{v \in S} X(F_{v}) \times \prod_{v \in T\setminus S}W_{v} \times \prod_{v \notin T} {\bf X}(\frak o_{v})]^{\br(X)}   $$
 via the inclusion $U \subset X$, this concludes the proof.
\end{proof}

\begin{lem}\label{imageeval}
 Let $F$ be a number field. 
Let $U \subset X$  be a  dense open set of a smooth geometrically integral $F$-variety $X$.
Assume $X(\A_{F})\neq \emptyset$. Let $\alpha_{1}, \dots, \alpha_{n} \in \br(X)$.
Let $S$ be a finite set of places of $F$.
The image of the evaluation map $U(\A_{F}^S) \to (\Q/\Z)^n$ defined by the sum of the invariants of 
each $\alpha_{i}$ on the  $U(F_{v})$ for $v \notin S$ coincides with the image
of the analogous evaluation map $X(\A_{F}^S) \to (\Q/\Z)^n$.
\end{lem}
\begin{proof} There is a natural map $U(\A_{F}^S) \to X(\A_{F}^S)$ which is compatible with
evaluation of elements of $\br(X)$, hence one direction is clear.
Let $\{M_{v}\} \in X(\A_{F}^S)$. There exist a finite set $T$ of places containing $S$ and
regular  integral models  ${\bf U} \subset {\bf X}$ of $U \subset X$ over $\frak o_{T}$ such that
 $\alpha_{i} \in \br({\bf X}) \subset  \br({\bf U}) $ for each  $i=1, \dots, n$, 
 such that $M_{v} \in {\bf X}(\frak o_{v})$ for each $v \notin T$,
and such that moreover ${\bf U}(\frak o_{v}) \neq \emptyset$ for $v \notin T$.
For $v \in T \setminus S$, let 
   $N_{v} \in U(F_{v}), v \in T \setminus S$ be close enough  to $M_{v} \in X(F_{v})$
that  $\alpha_{i}(N_{v})=\alpha_{i}(M_{v})$ for each $i=1,\dots, n$ (such points exist since $X$ is smooth).
For $v \notin T$, let $N_{v}$ be an arbitrary point of ${\bf U}(\frak o_{v})$.

Then
$$\sum_{v\notin S} \alpha_{i}(M_{v})= \sum_{v \in T, v \notin S} \alpha_{i}(M_{v})  = 
\sum_{v \in T, v \notin S} \alpha_{i}(N_{v}) = \sum_{v \notin S} \alpha_{i}(N_{v}).$$
\end{proof}

\begin{prop}
 Let $F$ be a number field. 
Let $U \subset X$  be a  dense open set of a smooth geometrically integral $F$-variety $X$.
Assume $X(\A_{F})\neq \emptyset$.

(i)  Assume $\br(X)/\br(F)$ finite. If $pr_{S}(X(\A_{F})^{\br(X)}) $ is strictly smaller than
$X(\A_{F}^S)$, then  $pr_{S}(U(\A_{F})^{\br(U)}) $ is strictly smaller than
$U(\A_{F}^S)$.

(ii) If $\br(X) \to \br(U)$ is an isomorphism, if $pr_{S}(U(\A_{F})^{\br(U)}) $ is strictly smaller than
$U(\A_{F}^S)$, then $pr_{S}(X(\A_{F})^{\br(X)}) $ is strictly smaller than
$X(\A_{F}^S)$.

\end{prop}

\begin{proof}
(i) Let $\alpha_{i} \in \br(X)$, $i=1,\dots,n$, generate $\br(X)/\br(F)$.

 If $pr_{S}(X(\A_{F})^{\br(X)}) $ is strictly smaller than
$X(\A_{F}^S)$, then there exists an ad\`ele $\{M_{v}\}  \in X(\A_{F}^S)$
such that for each  $\{N_{v}\} \in \prod_{v \in S}   X(F_{v})$ there exists $\alpha_{i}$
such that
$$\sum_{v \notin S} \alpha_{i}(M_{v}) + \sum_{v \in S} \alpha_{i}(N_{v}) \neq 0 \in \Q/\Z.$$
In other words, the image of the map
$ \prod_{v \in S} X(F_{v}) \to (\Q/\Z)^n$
given by
$ \{N_{v}\}  \mapsto \sum_{v \in S} \alpha_{i}(N_{v})$
does not contain   $\{ -\sum_{v \notin S} \alpha_{i}(M_{v}) \} \in (\Q/\Z)^n$.
By continuity, the image of the induced map $ \prod_{v \in S} U(F_{v}) \to (\Q/\Z)^n$
is the same as the image of the map $ \prod_{v \in S} X(F_{v}) \to (\Q/\Z)^n$.
By Lemma~\ref{imageeval}, there exists an ad\`ele $\{M'_{v}\} \in U(\A_{F}^S)$ such that:
$$\{ -\sum_{v \notin S} \alpha_{i}(M'_{v}) \} = \{ -\sum_{v \notin S} \alpha_{i}(M_{v}) \} \in (\Q/\Z)^n.$$
Thus for each $\{N'_{v}\} \in \prod_{v \in S}   U(F_{v})$ there exists some $i$ such that
$$\sum_{v \notin S} \alpha_{i}(M'_{v}) + \sum_{v \in S} \alpha_{i}(N'_{v}) \neq 0 \in \Q/\Z.$$
Hence $\{M'_{v}\} \in U(\A_{F}^S)$ does not belong to $pr_{S}(U(\A_{F})^{\br U})$.

\medskip

(ii) Let  $\{M_{v}\}  \in U(\A_{F}^S)$ be an ad\`ele such that
 for each  $\{N_{v}\} \in \prod_{v \in S}   U(F_{v})$ there exists $\alpha \in \br(U)$
such that
$$\sum_{v \notin S} \alpha(M_{v}) + \sum_{v \in S} \alpha(N_{v}) \neq 0 \in \Q/\Z.$$
The ad\`ele $\{M_{v}\}  \in U(\A_{F}^S)$ defines an ad\`ele $\{M_{v}\}  \in X(\A_{F}^S)$.
By hypothesis $\br(X) = \br(U)$. Thus for each $\alpha \in \br(X) = \br(U)$, the image of
the evaluation map on $U(F_{v})$ coincides with the image of the evaluation map
on $X(F_{v})$. We conclude that for each   $\{N_{v}\} \in \prod_{v \in S}   X(F_{v})$
there exists $\alpha \in \br(X)$
such that
$$\sum_{v \notin S} \alpha(M_{v}) + \sum_{v \in S} \alpha(N_{v}) \neq 0 \in \Q/\Z.$$
\end{proof}

\section{The easy   fibration method}\label{general}

\begin{prop}  \label{fib}
Let $F$ be a number field and  $f: X \to Y$ be a   morphism of smooth quasi-projective  geometrically integral varieties over $F$.
Assume that all 
geometric fibres of $f$ are non-empty and integral. Let $W \subset Y$ be a nonempty open set such that
 $f_{W}: f^{-1}(W) \to W$ is smooth.

Let  $S$ be a finite set of places of $F$. Assume

(i) $Y$ satisfies strong approximation off $S$.

(ii) The fibres of $f$ above $F$-points of $W$ 
 satisfy strong
approximation off $S$.

(iii) For each $v \in S$    the map $f^{-1}(W)(F_{v}) \to W(F_{v})$  is onto.

Then 
$X$ satisfies strong approximation off  $S$.
\end{prop}

\begin{proof}
There exist a finite set $T$ of places  containing all archimedean places and
a   morphism of smooth quasiprojective $\frak o_{T}$-schemes $\phi: \mathcal{X}  \to \mathcal{Y}$
which restricts to $f:X \to Y$ over $F$, and such that:

(a) All geometric fibres of $\phi$ are geometrically integral.

(b) For any closed point $m$ of $\mathcal Y$, the fibre at $m$, which is a variety over
the finite field $\kappa(m)$, contains a smooth $\kappa(m)$-point.

(c) For any $v \notin T$, the induced map $\mathcal{X}(\frak o_{v}) \to \mathcal{Y}(\frak o_{v})$ is onto.

The proof of this statement combines standard results from EGA IV 9
 and the Lang-Weil estimates
for the number of points of integral varieties over a finite field. Many variants have already
appeared in the literature.

\medskip

To prove the proposition, it is enough to show:

  {\it Given any finite set $T$ as above, with $S \subset T$, and  given, for each place $v \in T \setminus S$,
an open set $U_{v} \subset X(F_{v})$  such that  the open set
$$\prod_{v \in S} X(F_{v}) \times \prod_{v \in T\setminus S}U_{v} \times \prod_{v \notin T} {\bf X}(\frak o_{v})   $$
of $X(\A_F)$ is not empty, then  this set contains a point of the diagonal image   of $X(F)$ in $X(\A_{F})$.}

The Zariski open set $f^{-1}(W) \subset X$ is not empty.
For each  $v \in  T\setminus S$, we may thus replace $U_{v}$ by the  nonempty  open set $U_{v} \cap f^{-1}(W)(F_{v})$.
Since $f$ is smooth on $f^{-1}(W)$, $f(U_{v}) \subset Y(F_{v})$ is an  open set.
By hypothesis (i), in this open set
there exists a point $N \in Y(F)$ whose diagonal image lies in this open set.
 Let $Z=X_{N}=f^{-1}(N)$.
The point $N$ comes from a point $\bf N$ in $\mathcal{Y}(\frak o_{T})$.
The  $\frak o_{T}$-scheme $\mathcal{Z}:=\phi^{-1}({\bf N})$ is thus  a model of $Z$.
For $v \notin T$,  statement (c)  implies $\mathcal{Z}(\frak o_{v}) \neq \emptyset$.
By assumption (iii), we have $Z(F_v) \neq \emptyset$ for each $v \in S$.
For $v\in T \setminus S$, the intersection $U_v \cap Z(F_v) $ by construction  is a nonempty open set
of  $Z(F_v) $.  Assumption (ii) now guarantees that the product
$$ \prod_{v \in S}Z(F_v) \times \prod_{v \in T\setminus S} U_v \cap Z(F_v)  \times 
\prod_{v \notin T} \mathcal{Z}(\frak o_{v})$$
contains the diagonal image of a point of $Z(F)$. This defines a point in $X(F)$
which lies in the given  open  set of $X(\A_F)$.
\end{proof}

\bigskip

Let us recall a well known fact.

\begin{prop}\label{eichlerkneser}
Let $F$ be a  number field. Let $q(x_1, \dots, x_n)$ be a non-degenerate quadratic
form over  $F$ and let $c \in F^\times$. Assume $n \geq 4$.
 Let $X$ be the
 smooth affine quadric   defined by
$q(x_1, \dots, x_n)=c$.
Suppose  $X(F_{v}) \neq \emptyset$ for each real completion $F_{v}$.  
Then $X(F) \neq \emptyset$.
Let $v_{0}$ be a place of $F$ such that the quadratic form $q$
is isotropic at $v_{0}$.
Then $X$
 satisfies strong appproximation
off any finite set $S \subset \Omega_{F}$ containing $v_{0}$. 
\end{prop}

\begin{proof}
This goes back to Eichler and Kneser.
See \cite{CTX} Thm. 3.7~(b) and  Thm. 6.1.
\end{proof}

\begin{lem}\label{singularfibres}
Let $q(x_1, \dots, x_n)$ ($n \geq 1$)  be a non-degenerate quadratic
form over   a field $k$ of characteristic different from 2.
Let $p(t) \in k[t]$ be a nonzero polynomial.
Let $X$ be the affine $k$-scheme defined by the  equation
$q(x_1, \dots, x_n) = p(t).$
The singular points of  $X$ are the points defined by
$x_{i}=0$ (all $i$) and $t=\theta$ with $\theta$ a multiple root of $p(t)$.
In particular, if $p(t)$ is a separable polynomial, then $X$ is smooth over $k$. \qed
\end{lem}

 \begin{prop}\label{fourfinal}   Let $F$ be a number field and $X$ be an $F$-variety defined by an equation
$$q(x_1, \dots, x_n) = p(t) $$
 where $q(x_1, \dots, x_n)$ is a non-degenerate quadratic
form with $n\geq 4$ over $F$ and $p(t) \neq 0$ is a polynomial in $F[t]$.
Let $\tilde{X}$ be any smooth geometrically integral variety which contains
the smooth locus $X_{smooth}$ as a dense open set.
Assume  $X_{smooth}(F_{v}) \neq \emptyset$ for each real place $v$ of $F$.

(1)   $\tilde{X}(F)$ is Zariski-dense in $X$.

(2) $\tilde{X}$ satisfies weak approximation.

Let $v_{0}$ be a place of $F$ such that $q$ is isotropic over $F_{v_{0}}$.

(3)  
$\tilde{X}$ satisfies strong  approximation off any finite set $S$ of places
which contains $v_{0}$.
\end{prop}

\begin{proof} Statements (1) and (2), which are easy, are special cases of Prop.~3., p.~66 of \cite{CTSaSD}. Let us prove (3) for $\tilde{X}=X_{smooth}$, the smooth locus of $X$.
Let $f: X_{smooth}  \to  {\bf A}^1_{F}$ be given by the coordinate $t$. 
By Lemma \ref{strongSS'}, it suffices to prove the theorem for $S=\{v_{0}\}$.
Let $W$ be the complement of $p(t)=0$ in ${\bf A}^1_{F}$. 
Given Prop. \ref{eichlerkneser}, Lemma \ref{singularfibres},
statement (3) for $\tilde{X}=X_{smooth}$ is an immediate consequence of  Proposition \ref{fib} applied to the composite maps
$X_{smooth} \subset X \to {\bf A}^1_{F}$.
  Statement (3)  for an arbitrary $\tilde X$ is  then an immediate application of Proposition  \ref{strongUX}.
\end{proof}

\section{Algebra and arithmetic of the equation $q(x,y,z)=a$ } \label{basicquadric}

Let $q(x,y,z)$ be a nondegenerate quadratic form over a field $k$ of characteristic zero and let $a \in k^*$.
Let $Y/k$ be the affine quadric defined by the equation
$$q(x,y,z)=a.$$
This is an open set in the smooth projective  quadric defined by the homogenous equation
$$q(x,y,z)-au^2=0.$$
Let $d=-a.det(q)  \in k^{\times}$. 

\begin{prop}\cite[\S 5.6, \S 5.8]{CTX}\label{computbr}
Assume $Y(k)\neq \emptyset$.
If $d$ is a square, then $\br(Y)/\br(k)=0$. If $d$ is not a square, then
$\br(Y)/\br(k)=\Z/2$. For any field extension $K/k$,
the natural map $\br(Y)/\br(k) \to \br(Y_{K})/\br(K)$ is surjective.

(i) If $\alpha x+ \beta y + \gamma z +\delta =0$ is an affine equation for the tangent plane of $Y$
at a $k$-point of the projective quadric $q(x,y,z)-au^2=0.$
then the quaternion algebra $(\alpha x+ \beta y + \gamma z +\delta,d) \in \br(k(Y))$
belongs to $\br(Y)$ and it generates $\br(Y)/\br(k)$.

(ii) Assume $q(x,y,z)=xy-det(q)z^2$. Then the quaternion algebra $(x,d)  \in \br(k(Y))$
belongs to $\br(Y)$ and it generates $\br(Y)/\br(k)$. \qed
\end{prop}

 \begin{lem} \label{good} Let $F$ be a number field.
 Let $q(x_1, \dots, x_n)$ be a nondegenerate quadratic form over $F$.
  Let $v$ be a non-dyadic valuation of $F$. Assume $n \geq 3$.
 If the coefficients of $q(x_1, \dots, x_n)$ are in $\frak o_{v}$ and the determinant of $q(x_1, \dots, x_n)$ is a unit in $\frak o_{v}$, then for any $d \in \frak o_{v}$  the equation
$q(x_1, \dots, x_n)=d $ admits a solution $(\alpha_1,  \dots, \alpha_n)$ in $\frak o_{v}$ such that one of $\alpha_1, \dots, \alpha_n$ is a unit in $\frak o_{ v}^\times$.
\end{lem}
\begin{proof} This follows from  Hensel's lemma. \end{proof}

 \begin{lem}\label{prex}
Let $ v$  be a non-dyadic valuation of a number field $F$.  Let $q(x,y,z)$   be a  quadratic form defined over $\frak o_v$  with $v(det(q))=0$. 
Let $a \in \frak o_v$, $a \neq 0$.  Let ${\bf Y}$ be the $\frak  o_v$-scheme defined by the   equation
$$q(x,y,z)=a.$$ Let  $Y$ be the generic fibre of ${\bf Y}$ over $F_v$. Assume $ -a.det(q) \notin  F_v^{\times 2}$.
Let $${\bf Y}^\ast (\frak o_v)= \{  (x_v, y_v, z_v )\in {\bf Y}(\frak o_v): \ \text{one of} \ \ x_v, y_v, z_v  \in \frak o_v^\times  \}. $$
An element which represents the non-trivial element of $\br(Y)/\br(F_v)$
takes two values over ${\bf Y}^\ast (\frak o_v)$ if and only if
$v(a)$ is odd.
\end{lem}

\begin{proof} 
After an invertible $\frak o_v$-linear change of coordinates, one may write
  $$q(x,y,z)=xy-\det(q)z^2$$ over $\frak  o_v$. One then has 
$${\bf Y}^\ast (\frak o_v)= \{  (x_v, y_v, z_v )\in {\bf Y}(\frak o_v): \ \text{one of} \ \ x_v, y_v, z_v  \in \frak o_v^\times  \}.$$ 
By Proposition \ref{computbr},
one has $$\br(Y)/\br(F_v) \simeq \Z/2 \ \ \ \ \text{for} \ \
-a .\det(q)\notin F_v^{\times^2}$$ and the generator is given
by the class of the  quaternion algebra  $$(x,-a.det(q))\in
\br(F_v (Y)).$$

If $v(a)=v(- a .\det(q))$ is odd, one can choose $(x_v,y_v,0) \in {\bf Y}^\ast(\frak  o_v)$
where $x_v$ is a square resp. a non-square unit in $\frak o_v^\times$. On these points, $(x,-a. \det(q))$ takes
the value $0$, resp. the value $1/2$.

If $v(a)=v(- a .\det(q))$ is even,  we claim that any solution
$(x_v,y_z,z_v) \in {\bf Y}^\ast (\frak  o_v)$ satisfies that $ v(x_v)$ is even. Indeed, suppose there exists $(x_v, y_v, z_v)\in {\bf Y}^\ast (\frak  o_v)$ such that $v(x_v)$ is odd.
Then $y_v$ or $z_v$ is in $\frak o_v^\times$. If  we have $z_v\in \frak o_v^\times$, then by Hensel's lemma $-a.det(q) \in F_v^{\times^2}$,
which is excluded. We thus have $z_v\notin \frak o_v^\times$ and  $y_v \in \frak o_v^\times$. This implies $v(x_v y_v)$ is odd. Therefore
$$v(-det(q) .z_v^2)= v(a) < v(x_v y_v) $$ and $-a.det(q) \in F_v^{\times^2}$ by Hensel's lemma.
 A contradiction is derived and the claim follows. By the claim,
the algebra $(x,-a.det(q))  $  vanishes on ${\bf Y}^\ast (\frak o_v)$.
\end{proof}

\begin{lem} \label{infinit} Let  $k=F_{ v}$ be a completion of the number field $F$.
  Let $q(x,y,z)$ be a nondegenerate quadratic form over $k$
and let $a \in k^{\times}$.
Let $Y $ be the affine  $k$-scheme defined by the  equation
$$q(x,y,z)=a.$$
Assume $-a.\det(q) \notin
k^{\times 2}.$ Assume $Y$ has   a  $k$-point.
One has $\br(Y)/\br(k) \simeq \Z/2$.
Let $\xi$ be an element of $\br(Y)$ with nonzero image in
  $\br(Y)/\br(k)$. Then $\xi$ takes a single value over
$Y(k)$ if and only if $ v$ is a real place and $q  
$ is anisotropic over $F_ v$.
\end{lem}
\begin{proof}
By Proposition \ref{computbr},  
 one has
 $$\br(Y)/\br(k) \simeq \Z/2.$$ 
 Let $V$ be the quadratic space defined by $q(x,y,z)$ over $k$. 
 By \cite[p. 331]{CTX}, if we choose a lift $\xi \in \br(Y)$ of order 2,
 the image of the evaluation map
 $$Y(k) \to \br(k)$$
 coincides with the image of the composite map
 $$SO(V)(k)  \to k^{\times}/k^{\times 2} \to k^{\times}/N_{K/k}(K^{\times}),$$
 where $\theta: SO(V)(k)  \to k^{\times}/k^{\times 2}$ is the spinor map,
 $K=k(\sqrt{-a.det(q)})$ and $k^{\times}/k^{\times 2} \to k^{\times}/N_{K/k}(K^{\times})$
 is the natural projection, which is onto,  and   is
 by assumption an isomorphism if $k=\R$.
 For $k$ a nonarchimedean local field, the spinor map is surjective   \cite[91: 6]{OM}.
 For $k=\R$ the reals, the spinor map has trivial image in $\R^\times/\R^{\times 2} \simeq {\pm 1}$
 if and only if the quadratic form $q$ is anisotropic.  
 \end{proof}

The following proposition does not appear formally in \S 6 of \cite{CTX},
where attention is restricted to schemes over the whole ring of integers.
It follows however easily from  Thm. 3.7  and  \S 5.6 and \S 5.8 of \cite{CTX}.
\begin{prop}\label{deCTX}
Let $F$ be a number field. Let $Y/F$ be a smooth  affine quadric defined by an equation
$$q(x,y,z)=a.$$   Assume $Y(F)\neq \emptyset$.
Let $S$ be a   
finite set of places of $F$. Assume there exists $v_{0} \in S$
such that $q$ is isotropic over $F_{v_{0}}$.
Then  strong approximation  with Brauer-Manin obstruction off $S$ holds for $Y$.
Namely, the closure of the image of $Y(F)$ under the diagonal map
$Y(F) \to Y(\A_{F}^S)$ coincides with the image of 
$Y(\A_{F})^{\br(Y)} \subset Y(\A_{F})$ under the projection map $Y(\A_{F}) \to Y(\A_{F}^S)$.\qed
\end{prop}

\section{Computation of Brauer groups for the equation $q(x,y,z)=p(t)$}\label{computBrauer}

Let $k$ be a   field of characteristic zero,  $q(x,y,z)$ a nondegenerate quadratic form in 3 variables over $k$
and $p(t) \in k[t]$ a nonzero polynomial.

Let $X$ be the affine variety defined by the equation
\begin{equation} \label{equ} q(x,y,z)=p(t). \end{equation} 

The singular points of $X_{\ov k}$ are 
the points $(0,0,0,t)$ with $t$ a multiple root of $p$ (Lemma \ref{singularfibres}).
Let $U   \subset X_{smooth }$ be the  the complement of the closed set of $X$
defined by $x=y=z=0$.

 Let $\pi: \tilde{X} \to X$ a desingularization of $X$, i.e. $\tilde X$ is smooth
and integral,
the $k$-morphism $\pi $ is proper and birational. We moreover  assume that 
$\pi: \pi^{-1}(X_{smooth}) \to X_{smooth}$ is an isomorphism. In particular
$\pi: \pi^{-1}(U) \to U$ is an isomorphism.

Write  $p(t)=c.p_1(t)^{e_1}\dots p_s(t)^{e_s} $,
with  $c$ is in $k^\times$
and the  $p_i(t)$, $1\leq i\leq s,$ distinct monic irreducible   polynomials
over $k$. Let $k_{i}=k[t]/(p_{i}(t))$ for $1\leq i\leq  s$.

Let $K=\ov k(t)$ where $\ov k$ is an algebraic closure of $k$. The polynomial $p(t)$ is
a square in $K$ if and only if all the $e_{i}$ are even.

In this section we compute the Brauer groups of  $U$  and the Brauer group of the desingularization  $\tilde X$ of $X$. By purity for the Brauer group,  we have $\br(X_{smooth}) \oi \br(U$), and the
group $\br(\tilde{X})$ does not depend on the choice of  the resolution of singularities $\tilde X \to X$.

\begin{prop}\label{bg}
 Let  $p(t)=c.p_1(t)^{e_1}\dots p_s(t)^{e_s} $, $q(x,y,z)$ and
  $U$   be as above. 
If $p(t)$ is not a square in $K=\ov k(t)$, i.e. if not all $e_i$ are even, the natural map $\br (k) \to \br (U)$ is an isomorphism.
\end{prop}

\begin{proof}
Let $Z$ be the closed subscheme of ${\bf P}_k^3 \times {\bf A}_k^1$ defined by the equation
$$ q(x,y,z)=p(t)u^2$$
where $(x,y,z,u)$ are homogeneous coordinates for ${\bf P}_k^3$. Then $X$ can be regarded as an open set in $Z$ with $u\neq 0$. 
The complement of $X$ in $Z$ is given by $u=0$ and isomorphic to $D=C \times_{k} {\bf A}_k^1$  where $C$
is the projective conic in ${\bf P}_k^2$ defined by $q(x,y,z)=0$.
Let $f: {\bf P}_k^3 \times {\bf A}_k^1 \to {\bf A}_k^1$ be the projection onto ${\bf A}_k^1$.
We shall abuse notation and also denote by $f$ the restriction of $f$ to Zariski open sets of $X$.

Let $U_{\ov k}=U \times_k \ov k$. Let   $U_K=U \times_{{\bf A}_k^1} \Spec K$ and $Z_{K}=Z \times_{{\bf A}_k^1} \Spec K$.
Any invertible function on $U_{K} \subset Z_{K}$
has its divisor supported in $u=0$, which is an irreducible curve over $K$.
Hence such a function is a constant in $K^\times$.
Since the fibres of $f: U \to {\bf A}_k^1$ are nonempty, any invertible function on $U_{\ov k}$ is the inverse image
of a function in $K[U]^\times=K^\times$ which is invertible on ${\bf A}_{\ov k}^1$,
hence is in $\ov k^\times$. Thus $$ \ov k [U]^\times=\ov k^\times  . $$

Let $V=Z_{smooth}$   and $V_{\ov k}=V\times_k \ov k$. Since $p(t)$ is not a square in $K$, the 
 $K$-variety $$V_{K}=V_k\times_{{\bf A}_k^1}\Spec K \subset {\bf P}^3_{K}$$ is a smooth quadric
 whose discriminant is not a square.
 This implies that $\Pic(V_{K})=\Z$ and is spanned by the class
 of a hyperplane section.
 As a consequence, $\Pic(U_{K})=0$.
Since $U$ is smooth, $\Pic({\bf A}^1_{\ov k})=0$ and all the fibres of $f: U \to {\bf A}_k^1$ are
 geometrically integral, the restriction map
 $\Pic(U_{\ov k}) \to \Pic(U_{K})$ is an isomorphism.
 Thus $$\Pic(U_{\ov k})=0.$$

For any smooth projective quadric, the natural map $\br (K) \to \br (V_{K})$ is onto.
 Since $\br (K)=0$ by Tsen's theorem, one obtains $\br (V_{K})=0$. Moreover, since $V_{\ov k}$
is regular, 
the natural map $\br(V_{\ov k}) \to \br(V_{K})$ is injective. Therefore $ \br(V_{\ov k})=0.$

Let $$C_{\ov k}=C\times_k \ov k \ \ \ \text{ and } \ \ \ D_{\ov k}=D\times_k \ov k . $$
Since $D=C \times_{k} {\bf A}_k^1$ and $C_{\ov k} \simeq {\bf P}^1_{\ov k}$,
we have $H^1_{\et}(D_{\ov k}, \Q/\Z)=0$. 
Since $D_{\ov k}$ is a smooth divisor in the smooth variety $V_{\ov k}$,
we have the exact localization sequence
$$ 0 \to \br(V_{\ov k}) \to \br(U_{\ov k}) \to H^1_{\et}(D_{\ov k}, \Q/\Z).$$
One concludes
$$ \br(U_{\ov k})=0.$$

The Hochschild-Serre spectral sequence 
for \'etale cohomology of the sheaf ${\bf G}_{m}$ and  the projection morphism $U \to {\rm Spec} k$ yields a long exact sequence
$$ \Pic(U_{\ov k})^g \to  H^2(g,\ov k[U]^{\times}) \to Ker[ \br(U) \to \br(U_{\ov k})] \to H^1(g,\Pic(U_{\ov k}))$$
where $g=\Gal(\ov k/k)$. Combining it with  the displayed isomorphisms,  
we get
$$\br(k) \simeq \br(U).$$
  \end{proof}

Let us now consider  the case where  $p(t)$ is a square in $K=\ov k(t)$.

\begin{prop}\label{bexce}   Let  $p(t)=c.p_1(t)^{e_1}\dots p_s(t)^{e_s} $, $q(x,y,z)$ and 
  $U$  be as above. 
Assume all $e_i$ are even, i.e. $p(t)=c.r(t)^2$ with $c\in k^\times$ and $r(t) \in k[t]$  nonzero.
Let $d=-c.det(q)$.

The following conditions are equivalent:

(i) $d$ is not a square in $k$ and the natural map
$ H^3_{\et}(k,{\bf G}_m) \to H^3_{\et}(U,{\bf G}_m)$ is injective;

(ii) $\br(U)/\br(k)=\Z/2$.

If they are not satisfied then $ \br (U)/\br(k)=0$.
\end{prop}

\begin{proof}  
We keep the same notation as that in the proof of Prop.~\ref{bg}.
Since $p(t)$ is a square in $K$, one has $\Pic(V_{K}) \cong \Z \oplus \Z$. Moreover, the Galois group $\Gal(\ov k/k)$ acts on $\Pic(V_K)$ trivially if $-c.\det(q)\in  k^{\times 2}$. Otherwise the action of the Galois group $\Gal(\ov k/k)$  on $\Pic(V_K)$ factors through $\Gal(L/k)$ with permutation action, where $L=k(\sqrt{-c.det(q)})$. This yields isomorphisms of
  Galois modules $\Pic(U_{K})\cong \Z$ with trivial Galois action if $-c.\det(q)\in  k^{\times 2}$ and with the action of $\Gal(L/k)$ which sends a generator to its opposite otherwise. By the same argument as those in the proof of Prop.~\ref{bg}, one still has $${\ov k}^{\times} = \ov k [U]^{\times}, \ \ \ \Pic( U_{\ov k}) \simeq \Pic(U_{K}) \ \ \ \text{ and } \ \ \  \br(U_{\ov k})=0.$$
  
  Using $ \br(U_{\ov k})=0$, we deduce from the
 Hochschild-Serre spectral sequence  
 a long exact sequence
$$
  \br(k) \to \br(U) \to H^1(g, \Pic( U_{\ov k})) \to H^3_{\et}(k,{\bf G}_m) \to H^3_{\et}(U,{\bf G}_m).
$$

If $d=-c.\det(q)\in k^{\times 2}$, one has $$H^1(g, \Pic( U_{\ov k}))=\Hom_{cont}(g, \Z)=0$$  and the long exact sequence   yields $\br (U)/\br(k)=0.$

Assume $d=-c.\det(q)\notin  k^{\times 2}$. 
From
$$ H^1(g, \Pic( U_{\ov k}))=H^1(\Gal(L/k), \Pic (U_L))=\Z/2$$ where $U_L=U\times_k L$, 
one gets an inclusion $\br(U)/\br(k) \subset \Z/2$, 
which  is an equality  if and only if $ H^3_{\et}(k,{\bf G}_m) \to H^3_{\et}(U,{\bf G}_m)$ is injective.
\end{proof}

\begin{rem} \label{most}
 {\rm 
 The natural  map 
 $ H^3_{\et}(k,{\bf G}_m) \to H^3_{\et}(U,{\bf G}_m) $ is injective 
  under each of the following hypotheses: 
  
  (i) the open set $U$ has a $k$-point,
  
  (ii) the field $k$ is a number field (in which case $H^3_{\et}(k,{\bf G}_m)=0$).
    }
\end{rem}

\begin{prop}\label{bexcebis}  Keep notation as in Proposition \ref{bexce}.
Assume $\br(U)/\br(k)=\Z/2$. Then:
 
 (a) For any field extension $L/k$ the natural map $\br(U)/\br(k) \to \br(U_L)/\br(L)$ is onto.
 
 (b) For any field extension $L/k$ and any $\alpha \in {\bf A}^1(L)$ such that $p(\alpha)\neq 0$,
 the evaluation map  $\br(U)/\br(k) \to \br(U_{\alpha})/\br(L)$ on the fibre $q(x,y,z)=p(\alpha)$
 is onto.
\end{prop}

\begin{proof} 
 The long exact sequence  
$$
  \br(k) \to \br(U) \to H^1(g_{k}, \Pic( U_{\ov k})) \to H^3_{\et}(k,{\bf G}_m) \to H^3_{\et}(U,{\bf G}_m).
 $$
 is functorial in the base field $k$.
The assumption $\br(U)/\br(k)=\Z/2$ and the possible Galois action of the Galois group
on $\Pic( U_{\ov k})$ imply that the map $\br(U)/\br(k) \to H^1(g_{k}, \Pic( U_{\ov k}))$
is an isomorphism.

(a) If $d$ is a square in $L$,  then $\br(U_{L})/\br(L)=0.$ 
Assume not.  Let $L/k$ be any field extension.
Let $\overline{L}$ be an algebraic closure of $L$ extending $k \subset \overline{k}$.
If $\br(U_{L})/\br(L) =0$ the assertion is obvious.
If $\br(U_{L})/\br(L) =\Z/2$, then $H^1(g_{L}, \Pic( U_{\ov L}))=\Z/2$.
The natural map $\Pic( U_{\ov k}) \to \Pic( U_{\ov L})$
is  then a map of rank one lattices with nontrivial, compatible Galois actions.
This implies that the   natural map $H^1(g_{k}, \Pic( U_{\ov k})) \to H^1(g_{L}, \Pic( U_{\ov L}))$
is   an isomorphism. This implies that $\br(U)/\br(k) \to \br(U_L)/\br(L)$ is an isomorphism,
as claimed in (a).

\smallskip

  (b)  If $d$ is a square in $L$,  then $\br(U_{\alpha})/\br(L)=0.$
Assume not.  Let $\overline{L}$ be an algebraic closure of $L$ extending $k \subset \overline{k}$.
By the functoriality of the Hochschild-Serre spectral sequence for the morphism $U_{\alpha} \to U$, we have
a commutative diagram of exact sequences
\begin{equation}
\begin{CD}
 \br(k) &  \to  &\br(U) & \to H^1(g_{k}, \Pic( U_{\ov k})) & \to &H^3_{\et}(k,{\bf G}_m) & \to H^3_{\et}(U,{\bf G}_m)\\
  \downarrow &     &  \downarrow  &    \downarrow  &   & \downarrow &   \downarrow \\
   \br(L) &  \to  &\br(U_{\alpha}) & \to H^1(g_{L}, \Pic( U_{\alpha, \ov L}))& \to &H^3_{\et}(L, {\bf G}_m) & \to H^3_{\et}(U_{\alpha},{\bf G}_m)\\
\end{CD}
\end{equation} 
One readily verifies that the evaluation map  $\Pic( U_{\ov k}) \to  \Pic( U_{\alpha, \ov L})$ is an isomorphism
of Galois modules (split by a quadratic extension),  hence the map
$H^1(g_{k}, \Pic( U_{\ov k})) \to H^1(g_{L}, \Pic( U_{\alpha, \ov L}))$ is an isomorphism $\Z/2 = \Z/2$.
From the diagram we conclude that  $\br(U) \to \br(U_{\alpha}) /\br(L)$ is onto.
\end{proof}

\begin{prop}\label{cexp}  
Let  $p(t)=c. \prod_{i\in I}p_i(t)^{e_i} $, $q(x,y,z)$, $X$, $U$ and   
 $\pi: \tilde{X} \to X$ be as above.
Assume  $ H^3_{\et}(k,{\bf G}_m) \to H^3_{\et}(U,{\bf G}_m)$ is injective.
 Let $d=-c. det(q)$.

Consider the following conditions:

(i) All $e_i$ are even, i.e. $p(t)=c. r(t)^2$ for $c \in F^{\times}$ and some $r(t) \in k[t]$.

(ii) $d  \notin k^{\times 2}$.

(iii) For each $i \in I$, $d  \in k_{i}^{\times 2}$.

We have:

(a) If (i) or (ii) or (iii)  is not fulfilled, then $\br({\tilde X})/\br(k)=0$.

(b)  Assume $U(k) \neq \emptyset$. If   (iii) is fulfilled, then $\br({\tilde X}) \oi \br(U)$.

(c) If (i), (ii) and (iii) are fulfilled, then $\br({\tilde X})/\br(k) \oi \br(U)/\br(k)=\Z/2$.
In that case, for any field extension $L/k$ and any  $\alpha \in L$
such that $p(\alpha) \neq 0$, the evaluation map $\br(\tilde{X})/\br(k) \to \br(X_{\alpha})/\br(L)$
is surjective.

\end{prop}

\begin{proof} 
One has $\br(\tilde{X}) \subset \br(U)$.

\smallskip

{\it Proof of (a)}

By Proposition \ref{bg}, resp.  Proposition \ref{bexce},
if (i), resp. (ii), is not fulfilled, then $\br(U)/\br(k)=0$. 
Assume (i) and (ii) are fulfilled.  Proposition \ref{bexce} then gives
 $\br(U)/\br(k) \simeq \Z/2$.

 Let $F$ be the function field of the smooth projective conic $C$ defined by $q(x,y,z)=0$. 
 Assume (iii) does not hold. Let $i \in I$ such that $d  \notin k_{i}^{\times 2}$.
 Let $F_{i}$ be the composite field $F.k_{i}$.  Since $k$ is algebraically closed in $F$,
 so is $k_{i}$ in $F_{i}$. Thus $d$ is not a square in $F_{i}$.

By the same argument as in Proposition \ref{bexcebis}, the maps
 $$\Z/2=  \br(U)/\br(k)  \to \br(U_{F_{i}} )/\br(F_{i})$$
is an isomorphism.
 Over the field $F_{i}$, one may rewrite the equation of $X_{F_{i}}$ as
 $$ xy- det(q)z^2= c.r(t)^2$$
 and assume that $t=0$ is a root of $r(t)$.
 After restriction to the generic fibre of $U_{F_{i}} \to \Spec F_{i}[t]$, 
 the quaternion algebra   $(x,d) \in \br(F_{i}(X))$ defines a generator
 modulo $\br(F_{i} (t))$. This follows from    Proposition \ref{computbr}.
 Now  $(x,d)=(y.det(q),d)$ is unramified on the
 complement of the closed set $\{x=y=0\}$ on $U_{F_{i}}$, of codimension 2 in $U_{F_{i}}$, thus $(x,d)$ belongs
 to $\br(U_{F_{i}} )$. It thus generates $\br(U_{F_{i}} )/\br(F_{i})$.

  Define $h(T)\in k[T]$ by $Th(T)=cr(T)^2$.
   Consider the morphism
  $$ \sigma:  \Spec F_{i}[[T]] \to X$$
  defined by $$(x,y,z,t)=(T,h(T),0,T).$$ 
  The induced morphism $\Spec F_{i}((T)) \to X$ has its image in $U$. 
  It thus defines a morphism $\sigma: \Spec F_{i}((T)) \to \tilde{X}$.
  Since 
  $\pi: \tilde{X} \to X$ is proper, we conclude that the morphism
  $\sigma$ lifts to a morphism $ \tilde{\sigma}: \Spec F_{i}[[T]]) \to    \tilde{X}$.
  Suppose $(x,d) \in \br(U_{F_{i}})$ is in the image of $\br({\tilde X}_{F_{i}})Ê\to \br(U_{F_{i}})$.
  Then $\tilde{\sigma}^*((x,d))= (T,d)$ belongs to $\br(F_{i}[[T]])$. 
  But the residue of $(T,d) \in \br(F_{i}(T))$ at $T=0$ is $d  \neq 1 \in F_{i}^{\times}/F_{i}^{\times 2}$.  
  This is a contradiction. Taking into account Proposition \ref{bexcebis},  we conclude 
  that the embedding  $\br({\tilde X})/\br(k) Ê\hookrightarrow \br(U)/\br(k)=\Z/2$ is not onto,
hence $\br({\tilde X})/\br(k) =0$.
  
  \smallskip
  
  {\it Proof of (b)}
  
 Let $E=k(\sqrt{d})$. By Proposition \ref{bexce}, we have $\br(U_{E})/\br(E)=0$.
 Using the hypothesis $U(k) \neq \emptyset$, we see that any element of
 $\br(U) \subset \br(k(U))$ may be represented as the sum of an element of 
$\br(k)$ and the class of a quaternion algebra $(g,d)$  for some $g\in k(U)^\times$.

Assume 
(iii) is fulfilled.
Let $x$ be a point of codimension 1 of $\tilde{X}$ which does not belong to $p^{-1}(U)$.
Let $v$ be the associated discrete rank one valuation on the function field of $X$.
We then have 
$v(p_{i}(t))>0$ for some $i \in I$.
We thus have $k \subset k_{i} \subset \kappa_{v}$, where $\kappa_{v}=\kappa(x)$ is the residue field of $v$.
If  assumption (iii) is fulfilled we conclude that $d$ is a square in $\kappa_{v}$.

But then
 the residue of $(g,d)$
 at $x$, which is a power of $d$
 in $\kappa_{v}^{\times}/\kappa_{v}^{\times 2}$, is trivial. By purity for the Brauer group, we conclude
$\br({\tilde X})/\br(k)=\br(U)/\br(k)$. This proves (b).

\smallskip

{\it Proof of (c)}

This   follows from Prop. \ref{bexce} and Prop. \ref{bexcebis}.
 \end{proof}
 
 \bigskip

Let $Q$ be the smooth affine quadric over $k$ defined by $q(x,y,z)=c$.
For simplicity,  let us assume $Q(k)\neq \emptyset$. 
In the situation of Proposition \ref{bexce}, with $d=-c.\det(q)\notin (k^\times)^2$,
 one may  give an explicit generator in $\br(U)$ for $\br(U)/\br(k)=\Z/2$.

The assumption $Q(k)\neq \emptyset$ implies  $U(k)\neq \emptyset$. 
By
Prop. \ref{computbr},  we have
$ \br(Q)/\br(k) =
\Z/2  $.
Let  $ \alpha x+\beta y+\gamma z+ \delta=0$  define the tangent
plane of $Q$ at some $k$-point. Not all 
  $\alpha,\beta,\gamma$ are zero. As recalled in Prop. \ref{computbr},
  $$A=(\alpha x+\beta y+\gamma z+ \delta,d)\in \br(k(Q))$$
 \noindent belongs to $\br(Q)$  and generates $\br(Q)/\br(k)$.

Given a nonzero $r(t) \in k[t]$, let $W=Q\times_k ({\bf A}^1_{k} \setminus \{r(t)=0\}).$
Consider the birational $k$-morphism
$$ f: Q \times_k {\bf A}_k^1  \to X \subset {\bf A}^4_{k}; \ \ \ (x,y,z,t) \mapsto (r(t)x,r(t)y,r(t)z,t).$$
This map induces an isomorphism between $W=Q\times_k ({\bf A}^1\setminus \{r(t)=0\}$
and the   open set $V$
 of $U=X_{smooth}$ defined by $r(t)\neq 0$. Let $A_{V}$ be the image of $A$ inside $\br(V)$ under the composition map
$$ \br(Q) \to \br(W) \cong \br(V).$$

\begin{prop}\label{bexp} 
Let $p(t)=c. r(t)^2$ with $c\in k^\times$ and $r(t) \in k[t]$  nonzero.
Assume $$d=-c.\det(q)\notin k^{\times^2}. $$ 
Assume $Q(k) \neq \emptyset$. With notation as above, the element
$$B=A_{V}+(r(t),d)= (\alpha x + \beta y + \gamma z + \delta  r(t), d) \in \br(V)$$ 
can be extended to $\br(U)$ and it  generates the group $\br(U)/\br(k) \simeq \Z/2$.
\end{prop}

\begin{proof}  On $V \subset U=X_{smooth}$, we have
$$ A_{V} = (\alpha x/r(t) + \beta y/r(t) + \gamma z/r(t) + \delta , d)= (\alpha x + \beta y + \gamma z + \delta  r(t), d) - (r(t),d).$$
Thus 
$$B=A_{V}+(r(t),d)= (\alpha x + \beta y + \gamma z + \delta  r(t), d) \in \br(k(V))$$ 
is unramified on $V$. To check that it is unramified on $U$, it is enough to
compute the residue at the generic point of each of the components of $r(t)=0$ on $U$.
These are defined by a system $p_{i}(t)=0, q(x,y,z)=0$.
But at such a point, $\alpha x + \beta y + \gamma z + \delta  r(t)$  is a unit
since it
 induces  the class of $\alpha x + \beta y + \gamma z$ on the residue field, and this
 is not zero since  $\alpha x + \beta y + \gamma z$
  is not divisible by $q(x,y,z)$.
 Since  $d$ is clearly a unit, we conclude that $B$ is not ramified at such points.
The  natural map $\br(Q)/\br(k) \to \br(Q_{k(t)})/\br(k(t))$ is the identity on $\Z/2$.
It sends the nontrivial class $A$ to the class of $B$. The image of $B$
in $\br(U)/\br(k)=\Z/2$ is thus nontrivial. 
\end{proof}

\medskip

 One may use this proposition to give a more concrete description of
specialization of the Brauer group, as discussed in Prop. 5.7 (c).

\section{Arithmetic of the   equation $q(x,y,z)=p(t)$}\label{2dim}

Let $F$ be a number field,  $q(x,y,z)$ a nondegenerate quadratic form in 3 variables over $F$
and $p(t) \in F[t]$ a nonzero polynomial.
Let $X$ be the affine variety over $F$ defined by the equation
\begin{equation} \label{equ} q(x,y,z)=p(t). \end{equation} 
The singular points of $X_{\ov F}$ are 
the points $(0,0,0,t)$ with $t$ a multiple root of $p$ (Lemma \ref{singularfibres}).
Let $U \subset X_{smooth}$ be the complement of  the closed set of $X$ defined by $x=y=z=0$.

 Let $\pi : \tilde{X} \to X$ a desingularization of $X$, i.e. $\tilde X$ is smooth
and integral,
the map $\pi  $ is proper and birational. We    assume that 
$\pi: \pi ^{-1}(X_{smooth}) \to X_{smooth}$ is an isomorphism.
Thus  $\pi: \pi^{-1}(U) \to U$, is an isomorphism.

Write  $p(t)=c.p_1(t)^{e_1}\dots p_s(t)^{e_s} $,
with  $c$ is in $F^\times$
and the  $p_i(t)$, $1\leq i\leq s$ distinct monic irreducible   polynomials
over $F$. Let $F_{i}=F[t]/(p_{i}(t))$ for $1\leq i\leq  s$.

Under some local isotropy condition for $q$, we investigate  strong approximation for
the $F$-variety $\tilde X$.

This variety is equipped with an obvious fibration $\tilde{X} \to {\bf A}^1_{F}=\Spec F[t]$.
 
We begin  with two lemmas.

\begin{lem}\label{p} If $r(t)$ is an irreducible polynomial over a number field $F$, then
there are infinitely many valuations  $ v$ of $F$ for which there exist infinitely many
$t_{v}\in \frak o_v$ with $ v(r(t_{v}))=1$.
\end{lem}
\begin{proof} By Chebotarev's theorem, there are infinitely many
valuations $ v$ of $F$ which are totally split in the field $F[t]/(r(t))$.
Let $d$ denote the degree of $r(t)$.
For almost all such  $ v$, we may write  
$$r(t)=c\prod_{i=1}^d (t-\xi_i) \in F_{v}[t]$$  with all $\xi_i$ in $\frak o_{v}$
and $c$ and all   $\xi_i-\xi_j$ ($i\neq j$)
 units in $\frak o_{v}$. Since there are infinitely many elements of $\frak o_{v}$ with 
 $v$-valuation 1,
there exist infinitely many $t_{v}\in \frak o_v$ such that $ v(t_{v}-\xi_1)=1$. Then $ v(r(t_{v}))=1$.
\end{proof}

\begin{lem} \label{deci} 
Let $F$ be a number field, and $q(x,y,z)$ and $p(t)$ be as above.
If not all $e_{i}$ are even, then there exist infinitely many valuations $w$ of $F$
  for which there exists $t_{w} \in \frak o_w$ such that 
  $ w(p(t_{w}))$ is odd   and $-p(t_{w}).\det(q)\not \in
F_{w}^{\times 2}.$
\end{lem}
\begin{proof}

Assume  $e_{i_0}$ is odd for some
  $i_0 \in \{1,\cdots,s\}$. If $s=1$, the result immediately follows from Lemma \ref{p}.
  Assume $s>1$.
  
  For any $j\neq i_0$,  there are polynomials $a_j(t)$ and
$b_j(t)$ over $F$ such that
\begin{equation} \label{e} a_j(t)p_j(t)+b_j(t)p_{i_0}(t)=1 \end{equation}
holds.

  Let $S$ be a finite set of primes
such that each of the following conditions hold:

(i)  the coefficients of $q$ are integral away from $S$;

(ii) $ w(c)= w(\det(q))=0$ for all $w \not \in S$;

(iii) the coefficients of $a_j(t), b_j(t)$ for $j\neq i_0$ and
of $p_i(t)$ for $1\leq i\leq s$ are in $\frak o_w$ for all $w \not \in S$.

By applying Lemma \ref{p} to $p_{i_0}(t)$, we see that there exist infinitely many
primes $w \not \in S$ and $t_{w} \in \frak o_w$ such that
$w(p_{i_0}(t_{w}))=1 $.  By  equation (\ref{e}), one
has $w (p_j(t_{v}))=0$ for any $j\neq i_0$. This implies
$w(p(t_{w}))=e_{i_0}$ is odd. Therefore
$-p(t_{w})\cdot
\det(q)\not \in F_{w}^{ \times 2}. $
\end{proof}

\begin{prop}\label{threeHPWA}
   Let $F$ be a number field and $X$ be an $F$-variety defined by an equation
$$q(x,y,z) = p(t) $$
 where $q(x,y,z)$ is a non-degenerate quadratic
form   over $F$ and $p(t)$ is a nonzero polynomial in $F[t]$. 
Assume  $X_{smooth}(F_{v}) \neq \emptyset$ for each   place $v$ of $F$.
Then

(1)   $X_{smooth}(F)$ is Zariski-dense in $X$.

(2) $X_{smooth}$ satisfies weak approximation.

\end{prop}
\begin{proof}
This is a special case of Thm. 3.10, p. 66 of \cite{CTSaSD}.
\end{proof}

\begin{thm}\label{newmain}
Let $F$ be a number field. Let $ U \subset \tilde{X}$ be  as above.
Assume $U(\A_{F})\neq \emptyset$.
Let  $S$  be a finite subset of $\Omega_F$ which contains a place $v_0$ such that the quadratic form $q(x,y,z)$  is isotropic over $F_{v_0}$. 
Then strong approximation off $S$ with Brauer-Manin condition holds for any open set $V$ with
$U \subset V \subset \tilde{X}$, in particular for $X_{smooth}$.
\end{thm}

Since $\tilde{X}$ is smooth and geometrically integral,  
the hypotheses $U(\A_{F})\neq \emptyset$, 
$X_{smooth}(\A_{F})\neq \emptyset$  and $\tilde{X} (\A_{F})\neq \emptyset$
are all equivalent.

Taking into account the isomorphism  $\br(X_{smooth}) \oi \br(U)$, 
  the finiteness of  $\br(U)/\br(F)$ (\S 5) and Proposition \ref{strongBMUX},
 this theorem is an immediate consequence of the following  more  precise statement.

\begin{thm}\label{newmaindetail}
Let $F$ be a number field. Let  $p(t)=c.p_1(t)^{e_1}\dots p_s(t)^{e_s} $, $q(x,y,z)$,
$X$, $U$ and $\tilde{X}$ be as above. 
Let $d=-c.det(q)$. 
 Assume $U(\A_{F})\neq \emptyset$.
Let  $S$  be a finite subset of $\Omega_F$ 
which contains a place $v_0$ such that the quadratic form $q(x,y,z)$  is isotropic over $F_{v_0}$. 
 Assume $U(\A_{F})\neq \emptyset$. 
 
 Then $U(F) \neq \emptyset$ 
 is Zariski dense in $U$.
 
 (i) If at least one $e_{i}$ is odd, then $  \br(\tilde{X})/\br(F) = \br(U)/\br(F)=0$
 and strong approximation off $S$ holds for $U$ and for  $\tilde{X}$.
 
 (ii) If all $e_{i}$ are even and $d \in F^{\times 2}$, then
 $  \br(\tilde{X})/\br(F) = \br(U)/\br(F)=0$
  and strong approximation off $S$ holds for $U$ and  for $\tilde{X}$.
 
 (iii) If all $e_{i}$ are even
 and there exists $i$
 such that $d \notin F_{i}^{\times 2}$, then $  \br(\tilde{X})/\br(F)=0$,
 $\br(U)/\br(F)=\Z/2$,   
  strong approximation off $S$ with Brauer-Manin condition holds for $U$
  and for any open set $V$ with $U \subset V \subset \tilde{X}$,
  and strong approximation holds for any such open set $V$ which satisfies
  $\br(\tilde{X}) \oi \br(V)$.

 (iv) If all $e_{i}$ are even, $d \notin F^{\times 2}$, and for all $i$,
 $d \in  F_{i}^{\times 2}$, then $  \br(\tilde{X})/\br(F) = \br(U)/\br(F)=\Z/2$,
 and strong approximation off $S$ with Brauer-Manin condition holds for $U$
 and for $\tilde{X}$.
 
 (v) Strong approximation off $S$ fails for $U$, resp. for $\tilde{X}$,
 if and only if the following two conditions simultaneously hold:
 
 (a) $\br(U)/\br(F)=\Z/2$, resp. $\br(\tilde{X})/\br(F)=\Z/2$;
 
 (b)  $d$ is a square in $F_v$  for each  finite place
$v\in S$ and also  for each real place $v \in S$  such that 
either  $q(x,y,z)$ is  isotropic over $F_v$  or r(t) has a root over $F_v$.

\end{thm}

\begin{proof}

By Proposition \ref{threeHPWA},  $U(F)\neq \emptyset$ and $U(F)$ is Zariski dense in $U$.
The various values of $\br(U)$ and $\br(X)$ have been computed in \S  \ref{computBrauer}.
By Proposition  \ref{strongUX} and Proposition \ref{strongBMUX}, to prove (i) to (iv), it is  enough
to prove the
strong approximation statements (with Brauer-Manin obstruction) for $U$.

We fix a finite set $T$ of places, which   contains $S$, the infinite primes,
the dyadic primes and all the finite places $v$ where $q(x,y,z)$ has bad reduction.
We also assume that $p(t)$ has coefficients in $\frak o_{T}$ and that its leading coefficient $c$
is invertible in $\frak o_{T}$. We denote by $\bf X$ the $\frak o_{T}$-scheme given by
$$q(x,y,z)=p(t).$$
We let ${\bf U} \subset {\bf X}$ be the complement of the closed set defined by the ideal
$(x,y,z)$. 
We may extend $T$ so that there is a smooth integral  $\frak o_{T}$-scheme
$\bf{\tilde X}$ equipped with a proper birational $\frak o_{T}$-morphism $\bf{\tilde X} \to {\bf X}$ extending 
$\pi: \tilde{X} \to X$.

For any $v\notin T$,
 ${\bf U}({\frak o}_{v}) $
 is the set of points
$(x_{v},y_{v},z_{v}, t_{v})$ with all coordinates in ${\frak o}_{v}$,
$q(x_{v},y_{v},z_{v})=p(t_{v})$ 
 and one of $(x_{v},y_{v},z_{v})$ a unit.
 By Lemma \ref{good},  given any $t_{v} \in \frak o_{v}$, this set is not empty.

To prove the statements (i) to (iv), after possibly increasing $T$, we have to prove that
for any such finite set $T$ containing $S$,
a nonempty open set of $U(\A_{F})$ of the shape
$$ W_{U} = [\prod_{v\in S} U(F_{v}) \times \prod_{v \in T\setminus S} U_{v} \times
 \prod_{v \notin T} {\bf U}(\frak o_{v})]^{\br(U)}$$
with $U_{v} $ open in  $U(F_{v})$,
contains a point in $U(F)$.

Given $t_{0} \in \frak o_{T}={\bf A}^1({\frak o}_{T})$ with $p(t_{0})\neq 0$,
we let ${\bf U}_{t_{0}}/\Spec(\frak o_{T})$ 
be    the fibre of  ${\bf U}/{\bf A}^1_{{\frak o}_{T}}$ above $t_{0}$.
This is the $\frak o_{T}$-scheme defined by
$q(x,y,z)=p(t_{0}).$
We let $U_{t_{0}}={\bf U}_{t_{0}} \times_
{  \frak{o}_{T}   }    F$. 

It is enough to show that in each of the cases under consideration:

{\it There
exists $t_{0} \in  \frak o_{T}$ such that the set
$$[\prod_{v\in S} U_{t_{0}}(F_{v}) \times \prod_{v \in T\setminus S} U_{v} \cap U_{t_{0}}(F_{v}) \times \prod_{v \notin T} {\bf U}_{t_{0}}(\frak o_{v})]^{\br(U_{t_{0}})  }$$
is nonempty.}

Indeed, Proposition \ref{deCTX}   implies that  such a nonempty set  contains  an $F$-rational point.

 We have  $ \br(U)/\br(F) \subset \Z/2$.
If  $\br(U)/\br(F)$ is nonzero, we may represent the group by an element  $\xi$ of
order   $2$  in    $\br(U)$.
To prove the result, we may extend $T$. After doing so,  we may assume that $\xi$ vanishes identically on
each  ${\bf U}(\frak o_{v})$ for $v \notin T$.

We start with a
  point $\{M_{v}\}= \{(x_{v},y_{v},z_{z},t_{v})\}_{v\in \Omega_{F}}$
in $W_{U}$ 
such that $p(t_{v}) \neq 0$ for each $v \in \Omega_{F}$.

We have
$$ \sum_{v}\xi(M_{v})=0 \in \Z/2.$$

\medskip

In case (i), we choose a $w \notin T$ and a $t'_{w} \in \frak o_{w}$
with $w(p(t'_{w}))$ odd and  $-p(t'_{w}). det(q) \notin F_{w}^{\times 2}$.
The existence of such $w,t'_{w}$ is  guaranteed by Lemma 
\ref{deci}.

\medskip

Using the strong approximation theorem, we find a $t_{0} \in \frak o_{T}$
which is very close to each $t_{v}$ for $v \in T \setminus \{v_{0}\}$
and is also very close to $t'_{w}$ in  case (i).

  By Lemma \ref{good}, as recalled above, for each $v \notin S$, the projection
${\bf U}(\frak o_{v}) \to {\bf A}^1(\frak o_{v})$ is onto.
By assumption, $q$ is isotropic at $v_{0} \in S$,
hence $U(F_{v_{0}}) \to {\bf A}^1(F_{v_{0}})$ is onto.

Combining this with the implicit function theorem,
we find an ad\`ele $\{P_{v}\} \in  
 U_{t_{0}}  (\A_{F})  = {\tilde X}_{t_{0}}(\A_{F}) $
with the following properties:

$\bullet$ For $v \in T \setminus \{v_{0}\}$, $P_{v}$ is very close to $M_{v}$ in $U(F_{v})$, 
hence belongs to 
$U_{v} \cap {\bf U}_{t_{0}}(F_{v})$
 for $v \in T  \setminus S
$. Moreover $\xi(M_{v})=\xi(P_{v})$.

$\bullet$   For $v \notin T$, $P_{v} \in {\bf U}_{t_{0}}(\frak o_{v}) 
$, hence $\xi(P_{v})=0=\xi(M_{v})$.

 By the Hasse principle, there  exists an $F$-point
 on the affine $F$-quadric $U_{t_{0}}={\tilde X}_{t_{0}}$.

\medskip

Consider case (i). By the definition of $w$, $w(p(t_{0}))$ is odd, 
$-p(t_{0}). det(q) \notin F_{w}^{\times 2}$,
hence $-p(t_{0}). det(q) \notin F^{\times 2}$, thus
$\Z/2= \br(U_{t_{0}})/\br(F) \simeq \br(U_{t_{0},F_{w}})/\br(F_{w})$ by Proposition \ref{computbr}.
Let $\rho \in  \br(U_{t_{0}})$ be an element of order 2 which generates these groups.

If $\sum_{v} \rho(P_{v})=0$, the ad\`ele 
 $\{P_{v}\} \in  
 U_{t_{0}}   (\A_{F})$ belongs to the Brauer-Manin set of $U_{t_{0}}$.

Suppose $\sum_{v} \rho(P_{v})=1/2$. By Lemma \ref{prex} , $\rho$ takes two distinct values
on ${\bf U}_{t_{0}}(\frak o_{w})$.  We may thus choose a new point $P_{w} \in {\bf U}_{t_{0}}(\frak o_{w})$
such that now $\sum_{v} \rho(P_{v})=0$, that is the new ad\`ele 
$\{P_{v}\} \in  
 U_{t_{0}}   (\A_{F})$ belongs to the Brauer-Manin set of $U_{t_{0}}$, which completes the proof in this
 case.

\medskip

Consider case (ii).
In this case $-c.det(q).p(t_{0})^2 \in F^{\times 2}$, hence $\br(U_{t_{0}})/\br(F)=0$
by Proposition \ref{computbr}.  Thus the ad\`ele $\{P_{v}\} \in  
 U_{t_{0}}   (\A_{F})$ is trivially in the Brauer-Manin set of $U_{t_{0}}$,
 which completes the proof in this
 case.

\medskip

Let us consider (iii) and (iv). 
In these cases,  
 $-c.det(q)  \notin F^{\times 2}$, hence $-det(q).p(t) \notin  F(t)^{\times 2}$
 and $-det(q).p(t_{0}) \notin  F^{\times 2}$
 for any $t_{0} \in F$.  
 We have $\br(U)/\br(F)=\Z/2$ and $\br(U_{t_{0}})/\br(F)=\Z/2$ 
for any $t_{0}$ such that $p(t_{0}) \neq 0$.
The element $\xi \in \br(U)$ has now exact order 2. It  generates  $\br(U)/\br(F)$.
The restriction of this element to $\br(U_{t_{0}})/\br(F)=\Z/2$ is the generator
of that group (Prop. \ref{bexce} and \ref{bexcebis}).

By hypothesis,  $\sum_{v} \xi(M_{v})=0.$
We then have
$$ \sum_{v} \xi(P_{v})=  \xi(P_{v_{0}}) + \sum_{v \in T \setminus \{v_{0}\}} \xi(P_{v})
 =\xi(P_{v_{0}}) + \sum_{v \in T \setminus \{v_{0}\}} \xi(M_{v})= \xi(P_{v_{0}}) - \xi(M_{v_{0}}).$$
 
 If $d \in F_{v_{0}}^{\times 2}$, then  
 $\br(U_{F_{v_{0}}})/\br(F_{v_{0}})=0$ (Prop. \ref{bexce}) and
  $\xi(P_{v_{0}}) - \xi(M_{v_{0}})=0$. We thus get    $\sum_{v} \xi(P_{v})=0$.
  The ad\`ele $\{P_{v}\}$ is   in the
 Brauer-Manin set of $U_{t_{0}}$.

  Assume $d \notin F_{v_{0}}^{\times 2}$. Then 
  $\br(U_{F_{v_{0}}})/\br(F_{v_{0}}) \oi \br(U_{t_{0},F_{v_{0}}})/\br(F_{v_{0}}) =\Z/2$
  (Prop. \ref{bexce} and \ref{bexcebis}).
 The image of $\xi$ in $\br(U_{t_{0},F_{v_{0}}})/\br(F_{v_{0}})$ generates this group.
By Lemma \ref{infinit},
the class $\xi$ takes two distinct values on $U_{t_{0}}(F_{v_{0}})$.
This   holds whether $v_{0}$ is real or not,  because  by assumption $q$ is  isotropic at $v_{0}$.
We may then change
 $P_{v_{0}} \in 
 U_{t_{0}} (F_{v_{0}})  $ in order to ensure $\xi(P_{v_{0}}) - \xi(M_{v_{0}})=0$,
 which yields $ \sum_{v} \xi(P_{v})=0$. The ad\`ele $\{P_{v}\}$ is   in the
 Brauer-Manin set of $U_{t_{0}}$.
 
 This proves (iii) and (iv) for $U$.

  \medskip
 
It remains to establish (v). 

\medskip
 
Assume (a) and (b). Under (a), all $e_{i}$ are even and $d \notin F^{\times 2}$.
We let $\xi$ be an element of  exact order 2 in $\br(U)$, resp. $\br({\tilde X})$ which generates
$\br(U)/\br(F)$, resp. $\br({\tilde X})/\br(F)$.
 Under (b), at each finite place $v \in S$, by Proposition \ref{bexce} we have
$\xi_{F_{v}} \in \br(F_{v})$, hence $\xi$ is constant on $U(F_{v})$, resp. ${\tilde X}(F_{v})$.
The same holds at a real place $v $ such that $d \in F_{v}^{\times 2}$.
 At a real place $v \in S$ such that $d \notin F_{v}^{\times 2}$,
the form $q(x,y,z)$ is anisotropic over $F_{v}$ and $r(t)$ has no real root.
At such $v$, the equation after suitable transformation reads $x^2+y^2+z^2=(r(t))^2$
and $U(F_{v})=U(\R)$ is connected. Then $\xi$ is  constant on $U(\R)$.

Since $d \notin F^{\times 2}$ there are infinitely many finite nondyadic places $w \notin S$
such that $d \notin F_w^{\times 2}$.
Let $M$ be a point of $U(F)$, resp. $\tilde{X}(F)$,   with $p(t(M))\neq 0$.
Since $d \notin F^{\times 2}$ there are infinitely many finite places 
$w \notin S$ such that $d \notin F_w^{\times 2}$.  At such a place $w$, $\xi$
takes two distinct values on $U_{t(M)}(F_{w})={\tilde X}_{t(M)}(F_{w})$ (use Proposition \ref{bexcebis} and Lemma \ref{infinit}).
Pick   $P_{w} \in U_{t(M)}(F_{w})$ with $\xi(P_{w}) \neq \xi(M)_{F_{w}} \in \Z/2$.
If we let $\{P_{v}\}$ be the ad\`ele of $U$, resp. ${\tilde X}$ with $P_{v}=M$ for $v \neq w$
and $P_{w}$ as just chosen, then $\sum_{v}\xi(P_{v}) \neq 0$, and this ad\`ele lies in
 an open set of the shape
$\prod_{v\in S} U(F_{v}) \times \prod_{v \in T \setminus S} U_{v} \times \prod_{v\notin T} {\bf U}(\frak o_{v})$,
resp.
$\prod_{v\in S} {\tilde X}(F_{v}) \times \prod_{v \in T \setminus S} U_{v} \times \prod_{v\notin T} {\bf {\tilde X}}(\frak o_{v})$,
which contains no diagonal image of $U(F)$, resp. ${\tilde X}(F)$.
 Strong approximation off $S$ therefore fails for $U$, resp. $\tilde X$.
 
 \medskip
 
 Suppose either (a) or (b) fails. Let us prove that strong approximation holds off $S$.
  If (a) fails, then $\br(U)/\br(F)=0$, resp. $ \br(\tilde X)/\br(F)=0$,
 and we have   proved that strong approximation holds off $S$.
 We may thus assume $\br(U)/\br(F)=\Z/2$, resp. $ \br(\tilde X)/\br(F)=\Z/2$, hence
 all $e_{j}$ are even and $d \notin F^{\times 2}$,
 and that (b) fails. Then either
  
 (i)  there exists a finite place $v \in S$ with  $d \notin F_{v}^{\times 2}$
 
 or
 
 (ii) there exists a real place $v \in S$ with $d  \notin F_{v}^{\times 2}$, i.e. $d <0$,
 such that  $q$ is isotropic over $F_{v}$ or $r(t)$ has a root in $F_{v}$.
 
 We let $\xi$ be an element of  exact order 2 in $\br(U)$, resp. $\br({\tilde X})$ which generates
$\br(U)/\br(F)$, resp. $\br({\tilde X})/\br(F)$.  For any $t_{v}\in {\bf A}^1(F_{v})$ with $p(t_{v})\neq 0$,
$\xi$ generates $\br(U_{t_{v}})/\br(F_{v})$, resp. $\br({\tilde X}_{t_{v}})/\br(F_{v})$ (Prop. \ref{bexcebis}).
 If $v$ is a  finite place of $S$ with $d \notin F_{v}^{\times 2}$
then,  by Lemma \ref{infinit}, above any point of $t_{v } \in {\bf A}^1(F_{v})$ with $p(t_{v})\neq 0$,
 $\xi$ takes two distinct values on $U_{t_{v}}(F_{v})= {\tilde X}_{t_{v}}(F_{v})$. 
  It thus takes two distinct
 values on $U(F_{v})$, resp. ${\tilde X}(F_{v})$. 
 The same argument applies if $v \in S $ is a real place and $q$ is isotropic at $v$.
 If $v$ is a real place and $q$ is anisotropic at $v$, then one may write the equation of $X$
 over $F_{v}=\R$ as
 $$x^2+y^2+z^2= r(t)^2.$$
 The real quadric $Q$ defined by $x^2+y^2+z^2=1$ contains the point $(1,0,0)$. Applying the recipe
 in Proposition \ref{bexp}, one finds that the class of the quaternion algebra $(x-r(t),-1)$ in $ \br(F(U))$  lies
 in
   $\br(U)$
 and  generates $\br(U\times_{F}{\R})/\br(\R)$.
 By assumption, $r(t)$ has a real root. One 
 easily
  checks 
  that $(x-r(t))$ takes opposite
 signs on $U(\R)$ when one crosses such a real root of $r(t)$. Thus $\xi_{\R}=(x-r(t),-1)$ takes two distinct
 values on $U(F_{v})$.

 Let now $\{P_{v}\}$ be an ad\`ele of $U$, resp. $\tilde X$. 
   If $\sum_{v}\xi(P_{v})=1/2$, then we   change
 $P_{v}$ at a place $v \in S$ so that the new $\sum_{v}\xi(P_{v})=0$. 
 We then know that that we can approximate this family  off $S$ by
 a point in $U(F)$, resp. $X(F)$.
 \end{proof}
 
 \begin{rem}\label{watson}
 {\rm
   Over the ring of usual integers, a special case of Watson's Theorem 3 in \cite{Wat} reads as follows.

{\it  Assume the ternary quadratic form  $q(x,y,z)$ with integral coefficients is of rank 3
 over $\Q$ and isotropic over $\R$. Let $p(t) \in \Z[t]$ be a nonconstant polynomial.
 Assume
 
 (W) 
 For each big enough prime $l$, the equation $p(t)=0$ has a solution in the
local field $\Q_{l}$. 

  If the equation  
 $q(x,y,z)=p(t)$ has solutions in  
 $\Z_{l}$ for each prime $l$, then
it has a solution in~$\Z$.}

Let $k=\Q$ and $X/k$ and $\tilde X/k$ be as above. 
This result is a consequence of  Theorem \ref{newmaindetail}.
   Indeed, if
$\br({\tilde X})/\br(k)=0$, strong approximation holds for $\tilde X$, hence in particular
 the local-global
principle  holds for integral points of $\tilde X$.   By Proposition \ref{cexp},
   $\br({\tilde X})/\br(k)\neq 0$
 occurs only if all $e_{i}$ are even, $d \notin k^{\times 2}$
and $d \in  k_{i}^{\times 2}$ for all $i$. That is to say, for each $i$,
the quadratic field extension $k({\sqrt d})$ of $k$ lies in $k_{i}$.
There are infinitely many primes $v$ of $k$ which are inert in $k({\sqrt d})$.
For such primes $v$, none of the equations $p_{i}(t)=0$ admits a solution in $k_{v}$.
Condition (W) excludes this possibility.}
 \end{rem}

\section{Two examples}\label{twoexamples}

In this section we give two examples which exhibit a drastic failure of strong approximation:
there are integral points everywhere locally but there is no global integral point.

\medskip

The first example develops      \cite[(6.1), (6.4)]{Xu}.

\begin{prop}
Let $\bf X \subset \A^4_{\Z}$ be the  scheme over $\Bbb Z$ defined by $$-9x^2+2xy+7y^2+2z^2 =
(2t^2-1)^2. $$ 
Let $\bf U$ over $\Z$ be the complement of $x=y=z=0$ in $\bf X$.
Let $X={\bf X} \times_{\Z}\Q$ and $U={\bf U} \times_{\Z}\Q$.
Let ${\tilde X} \to X$ be a desingularization of $X$ inducing an isomorphism over $U$.
Let ${\bf {\tilde X}} \to {\bf X}$, with ${\bf U} \subset {\bf {\tilde X}}$,
be a proper morphism   extending $\tilde X \to X$.

Strong approximation off $\infty$ fails for $U$ and for $\tilde X$.
More precisely:

(i)  $$\prod_{p\leq \infty }{\bf X}(\Bbb Z_p) \neq
\emptyset \ \ \ \text{ and } \ \ \ {\bf X}(\Bbb Z)=\emptyset .  $$

(ii) $$\prod_{p\leq \infty }{\bf U}(\Bbb Z_p) \neq
\emptyset \ \ \ \text{ and} \ \ \ {\bf U}(\Bbb Z)=\emptyset .  $$

(iii) $$\prod_{p\leq \infty }{\bf {\tilde X}}(\Bbb Z_p) \neq
\emptyset \ \ \ \text{ and } \ \ \ {\bf {\tilde X}}(\Bbb Z)=\emptyset .  $$
\end{prop}

\begin{proof}
With  notation as in Theorem   \ref{newmaindetail}, we have
$F=\Q$, $v_0=\infty$, $S=\{v_0\}$. One has $det(q)=-2^7$ and
$d=-c.det(q)= 2^9$. We are in case (iv) of Theorem   \ref{newmaindetail}.
Over $\R$, $q(x,y,z)$ is isotropic.
By Theorem \ref{newmaindetail} (iv)   
we have $$\br({\tilde X})/\br(F)= \br(U)/\br(F) = \Z/2$$
and
by Theorem \ref{newmaindetail} (v)   we know that strong approximation off $S$  fails for $U$ and $\tilde X$.

The equation may be written as
\begin{equation}\label{eqaff} (x-y)(9x+7y)=2z^2-(2t^2-1)^2. \end{equation}

Let $Y/\Q$ be the smooth open set defined by 
 \begin{equation}\label{eqlisse} (x-y)(9x+7y)=2z^2-(2t^2-1)^2 \neq 0. \end{equation}
Thus $Y \subset U \subset X$.
We have  $Y(\Q)=U(\Q)=X(\Q)$ since $2$ is not a square in $\Q$.
We also have $Y(\Q_{p})=U(\Q_{p})=X(\Q_{p})$ for any prime $p$
such that $2$ is not a square in $\Q_{p}$.

On the 3-dimensional smooth variety $U$, 
the algebra 
 \begin{equation}\label{computeB}
 B=(y-x,2)=(-2(9x+7y),2)=(9x+7,2)
\end{equation}
is unramified off the codimension $2$ curve $x=y=0$,
hence by purity  it is unramified on $U$.
One could show by purely algebraic means that it generates  $\br(U)/\br(F) = \Z/2$
but this will follow from the arithmetic computation below.

Note that $U(\Q)=X(\Q)$, since the singular points of $X$ are not defined over $\Q$.

For $p\neq 2$, there is a point of ${\bf U}(\Z_{p})$ with $t=1$.
For $p\neq 3$, we have the point $(0,1/3,1/3,1)$ in ${\bf U}(\Z_{p})$.
Thus $\prod_{p\leq \infty }{\bf U}(\Bbb Z_p) \neq
\emptyset$.

For $p\neq 2$,  and $2$ not a square in $\Q_{p}$,
for any solution of  (\ref{eqlisse}) in $\Z_{p}$, $y-x$ and $9x+7y$ are 
$p$-adic units.
 For any $p \neq 2$, equality (\ref{computeB})  thus implies $B(M_{p})=0$ for any
point in ${\bf X}(\Z_{p}) \cap Y(\Q_{p})$. 
Since $U$ is smooth, $Y(\Q_{p})$ is dense in $U(\Q_{p})$.
Since ${\bf X}(\Z_{p})$ is open in $X(\Q_{p})$,
this implies that ${\bf X}(\Z_{p}) \cap Y(\Q_{p})$ is dense in ${\bf X}(\Z_{p}) \cap U(\Q_{p})$,
and then that $B(M_{p})=0$ for any
point in ${\bf X}^*(\Z_{p}):={\bf X}(\Z_{p}) \cap U(\Q_{p})$.
This last set contains ${\bf U}(\Z_{p})$.

The algebra $B$ trivially vanishes on ${\bf X}^*(\R):=U(\R)$.

 Let us consider a point $M_{2} \in  {\bf X}(\Z_{2}) \subset Y(\Q_{2})$.
From (\ref{eqlisse}), for such a point with coordinates $(x,y,z,t)$, we have
$$(x-y)(9x+7y) = \pm 1 \hskip2mm {\rm mod} \hskip2mm 8.$$
Thus the 2-adic valuation of  $y-x$ and of $9x+7y$ is zero. 
 If $B$ vanishes on $M_{2}$ 
then 
$y-x=1 \hskip2mm mod \hskip2mm 4$ and  $9x+7y=1 \hskip2mm  mod \hskip2mm 4$.
But then $16x=2  \hskip2mm mod \hskip2mm 4$, which is absurd.
Thus $B(M_{2})$ is not zero, that is $B(M_{2})=1/2 \in \Q/\Z$.

We conclude that for any point $\{M_{p}\} \in  \prod_{p} {\bf X}^*(\Z_{p}) \times X^*(\R)$,
$$\sum_{p} B(M_{p})=B(M_{2})= 1/2.$$

This implies ${\bf X}(\Z) =  {\bf X}(\Z) \cap  U(\Q) =  \emptyset$, hence 
 ${\bf U}(\Z)= \emptyset$ and ${\bf {\tilde X}} (\Z) = \emptyset$,
 since both sets map to ${\bf X}(\Z)$.

Since ${\bf {\tilde X}} \to {\bf X}$ is  proper, the
map ${\bf {\tilde X}}(\Z_{p}) \to {\bf X}(\Z_{p})$ contains $ {\bf X}^*(\Z_{p})$ in its image.
We thus have ${\bf {\tilde X}}(\Z_{p}) \neq \emptyset$. 

One actually has
$$ [\prod_{p\leq \infty} {\bf {\tilde X}}(\Bbb Z_p)]^{Br({\tilde X})} =\emptyset.$$ 
Indeed, the algebra $B=(y-x,2)$ on $U$ extends to an unramified class on $\tilde X$.
To see this, one only has to consider the points of codimension 1 on $\tilde X$
above the closed point $2t^2-1=0$ of $\A^1_{\Q}$. For the corresponding valuation $v$ on
the field $F(\tilde X)$, one have $v(2t^2-1)>0$, thus $2$ is a square in the
residue field of $v$, hence the residue of $(y-x,2)$ at $v$ is trivial.
 \end{proof}

\medskip

\bigskip

The next  example is inspired by an example of Cassels
(cf. \cite[8.1.1]{CTX}).

\begin{prop}
Let $\bf X \subset \A^4_{\Z}$ be the  scheme over $\Bbb Z$ defined by
$$x^2-2y^2+64z^2=(2t^2+3)^2.$$ 
Let $\bf U$ over $\Z$ be   the complement of $x=y=z=0$ in $\bf X$.
Let $X={\bf X} \times_{\Z}\Q$ and $U={\bf U} \times_{\Z}\Q$.
Let ${\tilde X} \to X$ be a desingularization of $X$.
Let ${\bf {\tilde X}} \to {\bf X}$ be a proper morphism extending $\tilde X \to X$.

Strong approximation off $\infty$ holds for  $\tilde X$ and fails for $U$.
More precisely:

(i) ${\bf {\tilde X}}(\Bbb Z)$ is dense in
$\prod_{p < \infty }{\bf {\tilde X}}(\Bbb Z_p)$.

(ii)  
  There are solutions
$(x,y,z,t)$ in $\Z$  and $p(t) \neq 0$, thus 
${\bf X}(\Z) \cap U(\Q) \neq \emptyset$.
  
  (iii) We have $\prod_{p\leq \infty }{\bf U}(\Bbb Z_p) \neq
\emptyset$  and $ [\prod_{p\leq \infty} {\bf U}(\Bbb Z_p)]^{Br(U)} =\emptyset, $
hence ${\bf U}(\Bbb Z)=\emptyset$ : there are no solutions
 $(x,y,z,t)$ in $\Z$  with $(x,y,z)$ primitive.
\end{prop}

\begin{proof}
With notation as in Theorem  \ref{newmaindetail}, we have
 $F=\Q$, $v_0=\infty$, $S=\{v_0\}$.
We have $d=2^7$. Over $\R$, $q(x,y,z)$ is isotropic.
We are in case (iii) of Theorem  \ref{newmaindetail}.
We have  $\br({\tilde X})/\br(F)=0$ and   $\br(U)/\br(F)=\Z/2$.

 According to Theorem  \ref{newmaindetail} (iii), strong approximation off $\infty$
 holds for $\tilde X$.

Theorem \ref{newmaindetail}  (v)
 then says that strong approximation off $S$  fails for $U$.
 That is, $U(\Q)$ is not dense in $ U(\A_{\Q}^{\infty})$.

The point $(x,y,z,t)=(3,0,0,0) \in U(\Q) \cap {\bf X}(\Z) $ provides a point in ${\bf U}(\Z_{p})$
 for each   prime $p\neq 3$ and for $p=\infty$. 
 In general, for $p$ odd, we have 
 ${\bf U}(\Bbb Z_p) \neq \emptyset$ by Lemma \ref{good}.

 Let us prove 
  Statement (iii).

Since $1-8z=0$ is the tangent plane on affine quadric $x^2-2y^2+64z^2=1$ over $\Q$  at the point $(0,0, \frac{1}{8})$, Proposition \ref{bexp} shows that
  $B= (2t^2 + 3-8z, 2)$ is the generator of $\Br(U)/\Br(F)$.
We have
\begin{equation}\label{factorizeeq}
 (2t^2+3-8z)(2t^2+3+8z)=x^2-2y^2
 \end{equation}
thus 
\begin{equation}\label{secondvalueB}
B=(2t^2 + 3-8z, 2)=(2t^2 + 3+8z, 2).
\end{equation}
Let $p$ be an odd prime such that $2$ is not a square modulo $p$.
For a point  $(x,y,z) \in {\bf U}(\Z_{p})$, if $p$ divides both $2t^2 + 3-8z$
and $2t^2 + 3+8z$, then on the one hand $p$ divides $z$ 
and on the other hand, by equation (\ref{factorizeeq}), it divides
$x^2-2y^2$, which then implies that $p$ divides $x$ and $y$.
Thus $p$ divides $x,y, z$, which is impossible for a point in ${\bf U}(\Z_{p})$.
We   conclude from (\ref{secondvalueB}) that for any odd prime $p$, $B$ vanishes on ${\bf U}(\Z_{p})$.

For $p=2$, for any   $t$ and $z$ in $\Z_{2}$, we have $2t^2 + 3-8z = \pm 3$ modulo $8$, hence
$$(2t^2 + 3-8z,2)=(\pm 3,2) = 1/2 \in \Br(\Q_{2}).$$

Thus
$$ [\prod_{p\leq \infty} {\bf U}(\Bbb Z_p)]^{\Br(U)} =\emptyset, $$ 
which implies  ${\bf U}(\Bbb Z)=\emptyset$.

 \end{proof}

 \section{Approximation for singular varieties}\label{singular}

 The following lemma is well known.

\begin{lem}
Let $k$ be a local field.    Let $X$ be a geometrically integral variety over $k$. 
Let $f: {\tilde X} \to X$ be a resolution of singularities for $X$, i.e. ${\tilde X}$ is a smooth,
geometrically integral $k$-variety and $f$ is a proper birational $k$-morphism.
The following  
closed 
subsets of $X(k)$ coincide:

(a) The closure of $X_{smooth}(k)$ in $X(k)$ for the topology of $k$.

(b) The set $f({\tilde X}(k)) \subset X(k)$.

In particular, this set, called the set of central points of $X$, does not depend
on the resolution $f: {\tilde X} \to X$. It will be denoted $X(k)_{\rm cent}$. 
\end{lem}

\begin{proof}
One uses the fact that for a nonempty  open set $U$  of $\tilde X$, $U(k)$
is dense in ${\tilde X}(k)$ for the local topology, and that
the inverse image of a compact subset of $X(k)$
under $f$ is a compact set in ${\tilde X}(k)$.
\end{proof}

\begin{definition}
Let $F$ be a number field.   Let $X$ be a geometrically integral variety over $F$.
Assume $X_{smooth}(F) \neq \emptyset$.
Let $S$ be a finite set of places of $F$. 
One says that $X$ satisfies 
central weak approximation at $S$ if  either of the following
conditions is fulfilled:

(a) $X_{smooth}(F)$ is dense in  $\prod_{v \in S} X_{smooth}(F_{v})$.

(b) $X_{smooth}(F)$ is dense in  $\prod_{v \in S} X(F_{v})_{\rm cent}$.

One says that $X$ satisfies weak approximation if this holds for any finite
set $S$ of places of~$F$.
\qed
\end{definition}

While discussing the possible lack of weak approximation for a given variety $X$
the natural Brauer-Manin obstruction is defined by means of the
Brauer group of a {\it smooth, projective} birational model  of $X$.

Let us now discuss strong approximation.

\begin{lem}
Let $F$ be a number field.    Let $X$ be a geometrically integral variety over $F$. 
Let $f: {\tilde X} \to X$ be a resolution of singularities for $X$, i.e. ${\tilde X}$ is a smooth,
geometrically integral $F$-variety and $f$ is a proper birational $F$-morphism.
Let $S$ be a finite set of places of $F$.
The following
closed 
subsets of $X( \Bbb A_F^S)$ coincide:

(a) The intersection of $X(\A_{F}^S)$ with $\prod_{v \notin S} X(F_{v})_{\rm cent}$.

(b)  The image of ${\tilde X}( \Bbb A_F^S) $ under $f: {\tilde X}( \Bbb A_F^S) \to X( \Bbb A_F^S)$.

This set does not depend
on the resolution $f: {\tilde X} \to X$. We shall call it the set of central $S$-ad\`eles of $X$, and we shall denote it
  $X( \Bbb A_F^S)_{\rm cent}$.
\end{lem}

\begin{proof}
There exists a finite set $T$ of places of $F$ containing $S$ and  a proper $\frak o_{T}$-morphism of
$\frak o_{T}$ schemes 
 ${\bf {\tilde X}} \to {\bf X}$ extending $\tilde X \to X$.
 For $v \notin T$,  one checks that  
 $${\bf {\tilde X}}(\frak o_{v}) = {\bf {X}}(\frak o_{v}) \times_{X(F_{v} )  } {\tilde X}(F_{v}).$$
\end{proof}

\begin{prop} Let $X$ be a geometrically integral variety over  the number field $F$. 
Assume   $X_{smooth}(F) \neq \emptyset$.
Let $f: {\tilde X} \to X$ be a resolution of singularities for $X$.
Let $S$ be a finite set of places of $F$.
The following conditions are equivalent:

(a) The diagonal  image of $X_{smooth}(F)$   in $X(\A_{F}^S)_{\rm cent}$ is dense.

(b) The diagonal  image of ${\tilde X}(F)$   in ${\tilde X}(\A_{F}^S)$ is dense.

  \end{prop}

\begin{definition}
If these conditions hold, we say that  central
  strong approximation holds for $X$ off  $S$.
  \end{definition}
  
  If central strong approximation off $T$ holds for $X$,   it holds off any
  finite set $S'$ containing $S$.

\begin{definition} Let $X$ be a geometrically integral variety over  the number field $F$. 
Assume   $X_{smooth}(F) \neq \emptyset$. Let $f:  {\tilde X}  \to X$ be a resolution of singularities.
Let $S$ be a finite set of places of $F$.
If  the diagonal  image of ${\tilde X}(F)$ in  $({\tilde X}(\Bbb A_{F}^S))^{\br({\tilde X})} \subset {\tilde X}(\Bbb A_{F}^S)$ is dense, 
we say that  central strong approximation with Brauer-Manin obstruction off $S$ holds for $X $.  If central strong approximation with Brauer-Manin obstruction off $S$ holds for $X$, it holds off  any
  finite set $S'$ containing $S$.
\end{definition}

We leave it to the reader to translate the statement in terms of $X( \Bbb A_F^S)_{\rm cent}$. 
We insist that the relevant group is the group $\br({\tilde X})$, which does not depend on the chosen
resolution of  singularities  ${\tilde X} \to X$.

\medskip

\begin{exa} Let $k$ be a local field 
and  $X$ be a $k$-variety defined by an equation
$$ q(x_1, \cdots, x_n) =p(t),$$
where $q$ is a nondegenerate quadratic form and $p(t) \in k[t]$ a nonzero polynomial.
Then $X(k)\neq X(k)_{cent}$ if and only if there is a zero $\alpha$ of $p(t)$ over $k$ of even order $r$ and the quadratic form in  $n+1$ variables $$q(x_1,\cdots,x_n)- p_0(\alpha)x_{n+1}^2$$ is anisotropic over $k$, where $p(t)=(t-\alpha)^r p_0(t)$.
\end{exa}

\begin{proof} 
By Lemma 3.3, a singular point of $X(k)$ is given by $(0, \cdots, 0, \alpha)$, where $\alpha$ is a zero of $p(t)$ of order $r>1$. Let $p(t)=(t-\alpha)^r p_0(t)$.
We may assume
$$q(x_1, \cdots, x_n) = \sum_{i=1}^n a_i x_n^2 .$$

Let $\pi$ denote a uniformizer of $k$ if $k$ is $p$-adic and some non zero element with $|\pi|<1$
when $k$ is archimedean.

Suppose $r$ is odd. Let $\alpha_l=\alpha+p_0(\alpha)a_1\pi^{2l} $, hence 
$ \lim_{l\rightarrow \infty} \alpha_l = \alpha.$   
For $l\gg 0$, one has $p_0(\alpha_l)=p_0(\alpha)\epsilon_l^2$ with $\epsilon_l \in k^\times$ 
and   $\epsilon_l \rightarrow 1$ as $l\rightarrow \infty$. 
Then $$P_l=(\epsilon_l a_1^{\frac{r-1}{2}} p_0(\alpha)^{\frac{r+1}{2}} \pi^{lr}, 0, \cdots, 0, \alpha_l)$$ are
  smooth points of $X(k)$ for $l\gg 0$ and $P_l\rightarrow (0, \cdots, 0, \alpha)$ when $l\rightarrow \infty$. 
  Therefore $(0, \cdots, 0, \alpha)\in X(k)_{cent}$.

Suppose $r$ is even and the quadratic form $q(x_1,\cdots,x_n)- p_0(\alpha)x_{n+1}^2$ is isotropic. 
There exists  $$(\theta_1, \cdots , \theta_n, \theta_{n+1})\neq (0, \cdots, 0, 0)$$ in 
$k^{n+1}$ such that $q(\theta_1, \cdots, \theta_n)=p_0(\alpha)\theta_{n+1}^2. $
If  $\theta_{n+1}= 0$, then the smooth points of $X(k)$
$$P_n=(\pi^l\theta_1, \cdots, \pi^l\theta_n, \alpha) \rightarrow (0,\cdots,0,\alpha) $$ as $l\rightarrow \infty.$
Therefore $(0, \cdots, 0, \alpha)\in X(k)_{cent}$.

If  $\theta_{n+1}\neq 0$, one can assume that $\theta_{n+1}=1$. Let $t_l=\alpha+\pi^{2l}.$ 
Then $p_0(t_l)=p_0(\alpha)\epsilon_l^2$ with $\epsilon_l \in k^\times$ and $\epsilon_l \rightarrow 1$ 
as $l\rightarrow \infty$. The smooth points of $X(k)$
$$P_n=(\pi^{rl}\epsilon_l\theta_1, \cdots, \pi^{rl}\epsilon_l\theta_n, t_l) \rightarrow (0,\cdots,0,\alpha) $$ as
 $l\rightarrow \infty$. Therefore $(0, \cdots, 0, \alpha)\in X(k)_{cent}$.

Suppose  $r$ is even  and the quadratic form in $n+1$ variables $$q(x_1,\cdots,x_n)- p_0(\alpha)x_{n+1}^2$$ is anisotropic over $k$.
Suppose  the singular point $P_0=(0, \cdots, 0, \alpha)$ is the  limit of a sequence of smooth $k$-points. There thus exists a sequence of smooth $k$-points $P_l, l\in \N,$ satisfying $P_l\rightarrow P_0$ when $l\rightarrow \infty$. Let $P_l=(Q_l,\alpha_l)$ where $\alpha_l$ is the $t$-coordinate of $P_l$. Then $p_0(\alpha_l)=p_0(\alpha)\epsilon_l^2 \neq 0$ with $\epsilon_l \in k^\times$ for $l\gg 0$. Therefore
$$q(Q_l)- p(\alpha_l)=q(Q_l)-p_0(\alpha) [(\alpha_l-\alpha)^{\frac{r}{2}}\epsilon_l]^2 =0$$ for $l\gg 0$, which implies that $q(x_1,\cdots,x_n)- p_0(\alpha)x_{n+1}^2$ is isotropic over $k$. A contradiction is derived,
the point $P_{0}$ does not lie in $X(k)_{cent}$.
\end{proof} 
We conclude that $X(k)\neq X(k)_{cent}$ may happen only  in the following cases.

1) The field  $k$ is $\Bbb R$ and $q(x_1, \cdots, x_n)$ is $\pm$-definite over $\Bbb R$ and there is a zero $\alpha$ of $p(t)$ over $\Bbb R$ of even order $r$ such that $p_0(\alpha)$ has $\mp$ sign.

2) The field  $k$ is $p$-adic field and  $n\leq 3$. One can determine if a quadratic space is anisotropic over $k$ by computing determinants and Hasse invariants, as  in \cite[42:9; 58:6; 63:17]{OM}.

\medskip

{\it Acknowledgements}.  Work on this paper was pursued during visits
of the authors in Beijing (October 2010) and at  Orsay (October 2011). For his visit to Beijing, J.-L. Colliot-Th\'el\`ene was supported by the  Morningside Center of Mathematics.
For his visit to Universit\'e Paris-Sud, Fei Xu   acknowledges support  from David Harari's ``IUF''~grant and from NSFC grant no. 11031004.


\end{document}